\documentclass[11pt, letterpaper]{article}
\usepackage[margin=1in]{geometry}
\usepackage{parskip}
\usepackage{algorithm}
\usepackage{algpseudocode}
\usepackage{tikz}
\usepackage[numbers]{natbib}
\usepackage{bm}

\usepackage{mathtools}
\usepackage{multicol}
\usepackage{xcolor}
\usepackage{amsmath, amssymb, amsthm}
\usepackage{subcaption}
\usepackage[autoplay=true, loop=true]{animate}
\usepackage{enumitem}
\usepackage{multirow, makecell}
\usepackage{tabularx}
\usepackage{array}
\newcolumntype{C}{>{\centering\arraybackslash}X}
\usepackage{doi}
\usepackage{hyperref}

\hypersetup{
    colorlinks=true,
    linkcolor=blue,
    filecolor=magenta,      
    urlcolor=blue,
    citecolor=blue
    }

\makeatletter
\newcommand\makebig[2]{%
  \@xp\newcommand\@xp*\csname#1\endcsname{\bBigg@{#2}}%
  \@xp\newcommand\@xp*\csname#1l\endcsname{\@xp\mathopen\csname#1\endcsname}%
  \@xp\newcommand\@xp*\csname#1r\endcsname{\@xp\mathclose\csname#1\endcsname}%
}
\makeatother

\makebig{biggg} {4.5}
\newcommand*\Bell{\ensuremath{\boldsymbol\ell}}
\newcommand{\Balpha}{\boldsymbol\alpha}
\newcommand{\Bbeta}{\boldsymbol\beta}
\newcommand{\Bgamma}{\boldsymbol\gamma}
\newcommand{\Bdelta}{\boldsymbol\delta}

\newcommand{\vol}[2]{{\rm{vol}_{\mbox{\tiny$#1$}}\left({\mbox{$#2$}}\right)}}

\DeclareMathOperator{\conv}{conv}
\DeclareMathOperator{\volume}{Vol}
\newcommand{\R}{\mathbb{R}}

\newtheorem{theorem}{Theorem} 
\newtheorem{proposition}{Proposition}
\newtheorem{lemma}{Lemma}

\newtheorem{definition}{Definition}

\theoremstyle{remark}

\title{Volume Formulae for the Convex Hull of the Graph of a Trilinear Monomial: A Complete Characterization for General Box Domains}
\author{Lillian Makhoul and Emily Speakman \\[0.5em] {\small Department of Mathematical and Statistical Sciences} \\[0.1em]
{\small University of Colorado Denver, Denver, CO USA} \\[0.5em] {\small \url{lillian.makhoul@ucdenver.edu} \quad \url{emily.speakman@ucdenver.edu}}}
\date{}

\begin{document}
\maketitle

\begin{abstract}
Solving difficult mixed-integer nonlinear programs via spatial branch-and-bound requires effective convex outer-approximations of nonconvex sets. 
In this framework, complex problem formulations are often decomposed into simpler library functions, whose relaxations are then composed to build relaxations of the overall problem. 
The trilinear monomial serves as one such fundamental library function, appearing frequently as a building block across diverse applications.
By definition, its convex hull provides the tightest possible relaxation and thus serves as a benchmark for evaluating alternatives.
Mixed volume techniques have yielded a parameterized volume formula for the convex hull of the graph of a trilinear monomial; however, existing results only address the case where all six bounds of the box domain are nonnegative. 
This restriction represents a notable gap in the literature, as variables with mixed-sign domains arise naturally in practice.
In this work, we close the gap by extending to the general case via an exhaustive case analysis. 
We demonstrate that removing the nonnegative domain assumption alters the underlying structure of the convex hull polytope, leading to six distinct volume formulae that together characterize all possible parameter configurations.\end{abstract}

\section{Introduction}\label{sec:Intro}

Mixed-integer nonlinear programs (MINLPs) are a broad class of optimization problems arising across numerous domains of science, engineering, and economics. The general framework combines the modeling power of nonlinear and potentially nonconvex functions with discrete decision variables, enabling formulations that can effectively capture many real-world complexities. For a comprehensive discussion, see \cite{BurerLetchford12, KronqvistBernalNeiraGrossman26}, and the references therein. 

Proving that a candidate solution is globally optimal for an MINLP is a fundamental challenge in computational optimization. Despite the inherent complexity, solvers such as \texttt{BARON} \citep{Sahinidis1996}, \texttt{COUENNE} \citep{Couenne08}, \texttt{ANTIGONE} \citep{Antigone2014}, \texttt{SCIP} \citep{SCIP2023}, and others, are able to globally optimize small to moderately sized instances. The primary algorithmic
framework underlying each of these solvers is spatial branch-and-bound (sBB), which requires a valid convex relaxation of the original problem that can be iteratively refined after branching. The dominant approach for constructing such relaxations is \textit{factorable programming} \citep{McCormick76}. In this setting, sBB operates via three key steps:
\begin{enumerate}
\item {\it Function Decomposition}: Complex functions in the problem formulation are assumed to be ``factorable"; that is, each function can be systematically decomposed against a given library of simple, often low-dimensional, functions into a directed acyclic graph (DAG). 
\item	{\it Convex Relaxations}: For every function in the given library, a method is available to construct a valid convex outer-approximation (i.e., relaxation) of the function's graph over a particular domain. Given an appropriate function decomposition, the convex relaxations of library functions are composed by introducing auxiliary variables to represent intermediate function values according to the DAG structure; see Figure \ref{fig:DAG} for an example. By composing relaxations, the algorithm constructs a valid relaxation of the original problem formulation which is used to bound the optimal objective value. 
\item	{\it Recursive Refinement}: Once initial bounds are obtained, sBB iteratively refines the solution space through ``spatial branching", dividing selected variable domains into smaller subregions. For each subregion, tighter convex relaxations are constructed, leading to sharper bounds on the objective function. By recursively applying this process, the algorithm narrows the gap between upper and lower bounds until some convergence criterion is met. 
\end{enumerate}
For a more detailed discussion of sBB, see  \cite{RYOO1995551,ryoo96,adjiman98,smith99}, and \cite{Tawarmalani2002}. 

\begin{figure}[h]
    \centering
    \includegraphics[width=0.55\linewidth]{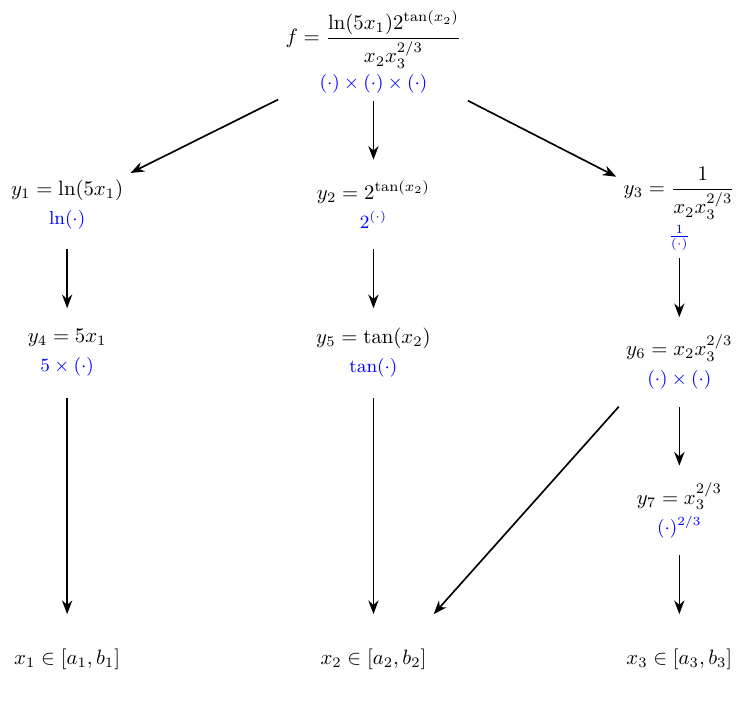}
    \caption{Example of a nonconvex function decomposed (factored) against the library functions shown. 
    }
    \label{fig:DAG}
\end{figure}

The efficiency of sBB for factorable programming fundamentally depends on the convexifications used to approximate the graphs of library functions. Problem relaxations are solved at every node of the branch-and-bound tree, and tighter convexifications yield stronger bounds on the objective value, allowing the algorithm to prune more of the search space and reduce overall tree size. However, a natural tradeoff exists between relaxation tightness and computational overhead. The convex hull may require complex inequality descriptions that are expensive to optimize over. At the same time, because convexifications are composed throughout the expression DAG, improvements at the library-function level benefit the entire model, while weak relaxations compound throughout the formulation.  

\subsection{Using volume to compare alternative convex relaxations}
Given that the choice of convexification method significantly impacts sBB performance, the analytic comparison of alternative methods is desirable. In the literature, the $(n+1)$-dimensional volume of the convex hull of a library function's graph has served as a baseline for such comparisons. This is intuitive: the convex hull, by definition, has the smallest volume among all convex relaxations. Thus, the closer in volume a relaxation is to the hull, the tighter its approximation. Experimental results in \cite{SpeakmanYuLee17} and \cite{LSS22} provide strong evidence that choosing library function convexifications with a smaller volume leads to improved objective function bounds when solving the overall problem relaxation.  

Volume-based comparison of mathematical optimization formulations was introduced by \cite{LeeMorris94}, who studied various combinatorial polytopes. See also \cite{Steingrímsson1994} for volume results on vertex-packing polytopes and \cite{KoLee97} for work on relaxations of the boolean-quadric polytope. More recently, \cite{AvisDevroye2026} study volume ratios between the cut polytope and several of its polynomial-time relaxations.  The volume between a convex lower bound and concave upper bound of a function over a bounded domain was formally defined as the ``convexity gap'' by \cite{CafieriEtAl12}, and used to guide the construction of compact relaxations of polynomial programs. When the tightest such bounds are used, this quantity coincides with the volume of the convex hull of the graph of the function.  More generally, the volume of polytopes is studied as a topic in its own right; see, e.g., \cite{Lawrence1991} and \cite{LasserreZeron2001}.

In the context of sBB for factorable functions, \cite{CafieriLeeLiberti2010} computationally investigated alternative relaxations of quadrilinear terms.  Beyond computational studies, volume has also been employed for the analytic comparison of low-dimensional sets. For instance,  \cite{LSS22, LSSX23} compared convex relaxations of simple disjunctions, including those obtained via the perspective function reformulation. 

The general multilinear monomial $g(x_1, \dots, x_n)\coloneq x_1x_2\dots x_n$ on a box domain $[a_1,b_1]\times[a_2,b_2]\times\dots\times[a_n,b_n]$ is a key substructure in many nonconvex
optimization problems; consequently, convex relaxations of its graph have been extensively studied. Relevant contributions include convex envelope characterizations \citep{rikun97, MeyerFloudas2005,JachMichaelsWeismantel2008}, and analysis of the strength of various alternative polyhedral relaxations \citep{RyooSahinidis2001, LuedtkeNamazafiLinderoth2012, DelPiaKhajavirad2021,Khajavirad2023, SchutteWalter2024}.

Within the literature, there has been a particular focus on the trilinear case of $n=3$, where the convex hull of the graph has been studied via both facet characterizations and volume analyses. \cite{meyer04mixed,meyer04pos/neg} provide a complete characterization of this polytope via its facets, including the case of mixed-sign domains. Volume analyses appear in \cite{SpeakmanLee17, SpeakmanYuLee17, LeeSkipperSpeakman18} and \cite{SpeakmanAverkov18}, but only under the assumption of nonnegative variable bounds.  Although the trilinear convex hull is fully described by the facet inequalities of \cite{meyer04mixed,meyer04pos/neg}, volume formulae for it remain valuable as a baseline for analytically comparing alternative convex relaxations. Such relaxations can be obtained by iteratively applying the classical McCormick inequalities \citep{McCormick76} for bilinear terms, yielding three so-called ``double McCormick" relaxations, corresponding to the three different ways of pairing the variables in the first application of McCormick. These are looser than the hull but simpler to work with, and volume formulae allow the resulting tradeoff to be quantified;  see \cite[Section 3]{SpeakmanLee17} for a fuller discussion of this perspective. For nonnegative domains, both the hull and the McCormick volumes are known in closed form, enabling such a comparison directly. The present paper takes the first step toward the analogous analysis for general domains by establishing the hull volume formula. 

More broadly, volume formulae can reveal structural insights of practical importance. The trilinear setting is well-suited for this kind of analysis because explicit descriptions are available for both the hull and its McCormick alternatives. For example, \cite[Corollary 4.3]{SpeakmanLee17} show that in the case of two binary variables and one continuous variable with strictly positive bounds, one double McCormick relaxation coincides exactly with the hull while the other two coincide and have arbitrarily larger volume as the continuous variable's upper bound grows. Insights of this kind may carry over to higher-order multilinear monomials, where complete facet characterizations are not generally available.

\subsection{Existing volume results for the trilinear monomial}

We consider the trilinear monomial $f(x_1,x_2,x_3)\coloneq x_1x_2x_3$ with $x_i\in[a_i,b_i],\ a_i<b_i$ for $i=1,2,3$.  Without loss of generality, we label the variables so that the conditions
   \begin{align*}a_1b_2\leq b_1a_2,\qquad a_1b_3\leq b_1a_3,\qquad \text{and} \qquad a_2b_3\leq b_2a_3 \label{eq:omega} \tag{$\Omega$}
   \end{align*}
\noindent are satisfied.

We define the convex hull of the graph of $f$ (parameterized by the $3 \times 2 =6$ variable bounds) as
$$\mathcal P_H^3(\mathbf{a},\mathbf{b})\coloneq\text{conv}\{(f,x_1,x_2,x_3)\in\mathbb R^4\ :\ f=x_1x_2x_3,\ x_i\in[a_i,b_i],\ i=1,2,3\}.$$

It is well-known that $\mathcal P_H^3(\mathbf{a},\mathbf{b})$ is polyhedral and has $2^3=8$ extreme points, see \cite{rikun97} for example.  Using this structure, \cite{SpeakmanLee17} construct a triangulation of $\mathcal{P}^3_H(\mathbf{a},\mathbf{b})$ and derive a formula for its volume in terms of the six parameters defining the box domain.  The result assumes all variable bounds are nonnegative. 
\begin{theorem}[Theorem 4.1 in \cite{SpeakmanLee17}]\label{thm:og-formula}
    Assume $0\leq a_i < b_i$ for $i=1,2,3$ and \eqref{eq:omega} holds. Then
    \begin{align*}
    \volume\left (\mathcal{P}_H^3(\mathbf{a},\mathbf{b})\right ) &= (b_1-a_1)(b_2-a_2)(b_3-a_3)\times\\
    &(b_1(5b_2b_3-a_2b_3-b_2a_3-3a_2a_3)+a_1(5a_2a_3-b_2a_3-a_2b_3-3b_2b_3))/24.
    \end{align*}
\end{theorem}

Under the assumption of nonnegative bounds, \cite{SpeakmanLee17} also derive volume formulae for each of the three double McCormick relaxations and rank the resulting volumes under the ordering \eqref{eq:omega}. \cite{SpeakmanAverkov18} re-prove Theorem~\ref{thm:og-formula} using tools from mixed volume theory. However, the framework does not extend straightforwardly to the double McCormick relaxations.

\subsection{Our contribution}
The existing volume results all assume nonnegative variable bounds, with the general mixed-sign case remaining open. This case is well-motivated even for MINLPs whose original formulation contains no trilinear monomials and/or uses only nonnegative variables. First, trilinear monomials can arise as auxiliary expressions during decomposition. Figure~\ref{fig:DAG} illustrates this: the decomposition of a non-trilinear function $f$ introduces an auxiliary trilinear monomial $y_1 y_2 y_3$, which itself requires a convex relaxation. Second, even when the original variables of an MINLP have nonnegative domains, the auxiliary variables introduced by decomposition may take both positive and negative values. For example, an auxiliary variable $y = \ln(x)$ takes negative values whenever $x \in (0, 1)$.

In the present work, we close this gap in the literature by extending Theorem~\ref{thm:og-formula} to general box domains via an exhaustive case analysis. Employing the mixed volume framework of \cite{SpeakmanAverkov18}, we derive explicit volume formulae for all parameter configurations.

\subsection{Organization}
The remainder of this paper is organized as follows. Sections \ref{sec:notation}--\ref{sec:prelim} provide the required preliminaries, beginning by defining our notation in Section \ref{sec:notation}. Section \ref{sec:background} reviews the necessary background from mixed volume theory, and Section \ref{sec:geometric} provides some useful geometric insight. Section \ref{sec:prelim} establishes the foundational results required for our main theorem, which is presented in Section \ref{sec:main}. We end with brief concluding remarks and suggestions for future work in Section \ref{sec:conclusion}. The \hyperlink{appendix}{appendix} contains the technical lemmas required for our main results, along with complete proofs.

\section{Notation}
\label{sec:notation}
We use standard notation for polytopes, including $\conv(\cdot)$ to denote the convex hull, and write $E(P)$ for the set of extreme points of a polytope $P$. For a convex body $B\subset \R^d$, $\vol{d}{B}$ denotes its $d$-dimensional volume (i.e., Lebesgue measure).  We write $\mathcal{K}^n$ to denote the set of all nonempty compact convex sets in $\mathbb{R}^n$.  Vectors are denoted in boldface (e.g., $\mathbf{v}$, $\mathbf{u}$).

Throughout this paper,  we work with box domains: $x_i \in [a_i, b_i]$ where $a_i < b_i$ for $i = 1, 2, 3$. In the literature, this notation (that is, characterizing the variable intervals by their lower and upper bounds) has been used extensively in prior work on the trilinear case, including facet descriptions of $\mathcal{P}^3_H$ \citep{meyer04pos/neg, meyer04mixed}, inequality descriptions of various natural relaxations \citep{SpeakmanLee17}, and related volume formulae \citep{SpeakmanLee17, LeeSkipperSpeakman18, SpeakmanAverkov18}. However, intervals can equivalently be characterized by their center point and half-length.

For $i = 1, 2, 3$, define
\begin{equation*}
c_i \coloneq \frac{a_i + b_i}{2} \quad \text{and} \quad \ell_i \coloneq \frac{b_i - a_i}{2},
\end{equation*}
so that $a_i = c_i - \ell_i$, $b_i = c_i + \ell_i$, and $x_i \in [c_i - \ell_i, c_i + \ell_i]$.

We then write
$$\mathcal P_H^3(\mathbf{c},\Bell)\coloneq\text{conv}\left\{(f,x_1,x_2,x_3)\in\mathbb R^4\ :\ f=x_1x_2x_3,\ x_i\in[c_i-\ell_i,c_i+\ell_i],\ i=1,2,3\right\}.$$
Substituting into \eqref{eq:omega} and simplifying yields an equivalent formulation of the variable labeling condition in terms of centers and half-lengths. Since $\ell_i > 0$ for all $i$, we obtain
\begin{align}
        \frac{c_1}{\ell_1}\leq \frac{c_2}{\ell_2}\leq \frac{c_3}{\ell_3}. \label{eq:omegaprime} \tag{$\Omega'$}
\end{align}

\noindent Using this notation, the result of \cite{SpeakmanLee17} can be restated as follows.
\begin{theorem}[Theorem 4.1 in \cite{SpeakmanLee17}, restated]\label{theorem:nonnegative restated}
Suppose $0<\ell_i\leq c_i\;$ for $i=1,2,3$ and \eqref{eq:omegaprime} holds. Then
$$\vol{4}{\mathcal{P}_H^3(\mathbf{c},\Bell)}=\frac{8}{3}\ell_1^2\ell_2^2\ell_3^2\left(\frac{c_1}{\ell_1}+2\frac{c_2}{\ell_2}+2\frac{c_3}{\ell_3}\right).$$
\end{theorem}
\vspace{4mm}

This center and half-length notation simplifies exposition and highlights the essential structure in our analysis. We use it throughout the remainder of the paper.

\section{Background}
\label{sec:background}

This section reviews the key concepts from mixed volume theory needed for our analysis. The presentation closely follows Section 2 of  \cite{SpeakmanAverkov18}, with proofs omitted. For complete details and proofs, see \cite{SpeakmanAverkov18} and \cite{Schneider13}.

\begin{theorem}[Theorem 2.1 in \cite{SpeakmanAverkov18}]
\label{thm:mixedvol}
    There is a unique nonnegative function $V:\underset{n}{\underbrace{\mathcal{K}^n\times\cdots\times\mathcal{K}^n}}\rightarrow\mathbb R$, the \textit{mixed volume}, which is invariant under permutation of its arguments, such that, for every positive integer $m>0$,  
    $$\vol{n}{t_1K_1+t_2K_2+\cdots+t_mK_m}=\sum_{i_1,\ldots,i_n=1}^mt_{i_1}\ldots t_{i_n}V(K_{i_1},\ldots,K_{i_n}),$$
    for arbitrary $K_1,\ldots,K_m\in\mathcal{K}^n$ and $t_1,t_2,\ldots,t_n\in\mathbb R_+$.
\end{theorem}
\vspace{4mm}

\begin{theorem}[Theorem 2.2 in \cite{SpeakmanAverkov18}]
\label{thm:mixedvol2}
    The \textit{mixed volume} function satisfies the following properties:
    \begin{enumerate}[label=(\roman*)]
        \item $\vol{n}{K_1}=V\left(K_1,\ldots,K_1\right)$,
        \item $V\left(t'K_1'+t''K_1'',K_2,\ldots,K_n\right)=t'V\left(K_1',K_2,\ldots,K_n\right)+t''V\left(K_1'',K_2,\ldots,K_n\right)$,
    \end{enumerate}
    for $K_1,K_2,\ldots,K_n\in\mathcal{K}^n$ and $t',t''\in\mathbb R_+$.
\end{theorem}
\vspace{4mm}

\begin{definition}[Support function]\label{def:support function}
    For a nonempty compact set $K\subseteq\mathbb R^n$, the support function $h_K:\mathbb R^n\rightarrow \mathbb R$ is defined by $$h_K(\mathbf{u}) = \sup_{\mathbf{x}\in K}\mathbf{x}^T\mathbf{u}.$$
\end{definition}
\vspace{4mm}

\begin{definition}[Set of outer facet normals]\label{def:outerfacetnromals}
    For a full-dimensional polytope, $P\subseteq \mathbb R^n$,
    \begin{align*}
    \mathcal{U}(P) \coloneq &\text{ the set of all outer facet normals, $\mathbf{u}$, such that the length of $\mathbf{u}$ equals} \\ &\text{ the $(n-1)$-dimensional volume of the corresponding facet.}
    \end{align*}
\end{definition}

When $n=3$ and $P\subseteq\mathbb{R}^3$ is a tetrahedron, we may explicitly derive its set of outer facet normals using the following result. 

\begin{lemma}[Lemma 2.7 in \cite{SpeakmanAverkov18}]\label{lemma:outerfacetnormals}
Let $M$ be a $4 \times 3$ matrix of the following form:
$$M \coloneq \begin{bmatrix}\mathbf{\Balpha} & \mathbf{\Bbeta} & \mathbf{\Bgamma} \\ 1 & 1 & 1 \end{bmatrix} = \begin{bmatrix}\alpha_1 & \beta_1 & \gamma_1 \\ \alpha_2 & \beta_2 & \gamma_2 \\\alpha_3 & \beta_3 & \gamma_3 \\1 & 1 & 1 \end{bmatrix},$$
 and define $D_i(\mathbf{\Balpha}, \mathbf{\Bbeta},\mathbf{\Bgamma})$ to be the determinant of the $3\times3$ matrix obtained by deleting the $i$-th row, $\begin{bmatrix}\alpha_i & \beta_i & \gamma_i\end{bmatrix}$, from $M$.   Let $T$ be a tetrahedron in $\mathbb R^3$, with vertices $\Balpha,\ \Bbeta,\ \Bgamma$, and $\Bdelta$, and without loss of generality, label the vertices such that
\begin{equation*}
      \text{det}\left[\begin{array}{cccc}
      \Balpha & \Bbeta & \Bgamma & \Bdelta\\
       1 & 1 & 1 & 1 
  \end{array}\right]> 0.\end{equation*}
  Then $$\mathcal{U}(T) = \frac{1}{2}\left\{ \left[\begin{array}{c}
       D_1(\Balpha,\Bbeta,\Bgamma)  \\
       -D_2(\Balpha,\Bbeta,\Bgamma) \\
       D_3(\Balpha,\Bbeta,\Bgamma)
  \end{array}\right], 
  \left[\begin{array}{c}
       -D_1(\Bdelta,\Balpha,\Bbeta)  \\
       D_2(\Bdelta,\Balpha,\Bbeta) \\
       -D_3(\Bdelta,\Balpha,\Bbeta)
  \end{array}\right], 
  \left[\begin{array}{c}
       D_1(\Bgamma,\Bdelta,\Balpha)  \\
       -D_2(\Bgamma,\Bdelta,\Balpha) \\
       D_3(\Bgamma,\Bdelta,\Balpha)
  \end{array}\right] ,  
  \left[\begin{array}{c}
       -D_1(\Bbeta,\Bgamma,\Bdelta)  \\
       D_2(\Bbeta,\Bgamma,\Bdelta) \\
       -D_3(\Bbeta,\Bgamma,\Bdelta)
  \end{array}\right] \right\}.$$
\end{lemma}
\vspace{4mm}

\begin{theorem}[Theorem 2.6 in \cite{SpeakmanAverkov18}]
\label{thm:polytopemixedvol}
     Let $P\subseteq \mathbb R^n$ be a full-dimensional polytope and let $K$ be a nonempty compact set in $\mathbb R^n$. Then the {mixed volume} of $(n-1)$ copies of $P$ and one copy of $K$ is given by: $$V(P,\ldots,P,K)=\frac{1}{n}\sum_{\mathbf{u}\in\mathcal{U}(P)}h_K(\mathbf{u}).$$
\end{theorem}

\section{Geometric insight} \label{sec:geometric}
The key geometric observation underlying the mixed volume approach becomes clear when viewing $\mathcal P_H^3(\mathbf{c},\Bell)$ as the convex hull of its extreme points.  Let
 {\small \begin{align*} 
&\mathbf{v}_1 \coloneq \begin{bmatrix}
      (c_1-\ell_1)(c_2-\ell_2)(c_3-\ell_3) \\
      (c_1-\ell_1)\\
      (c_2-\ell_2)\\
      (c_3-\ell_3)
\end{bmatrix},\; \mathbf{v}_2 \coloneq\begin{bmatrix}
      (c_1+\ell_1)(c_2-\ell_2)(c_3-\ell_3) \\
      (c_1+\ell_1)\\
      (c_2-\ell_2)\\
      (c_3-\ell_3)
\end{bmatrix}, \; \mathbf{v}_3\coloneq\begin{bmatrix}
      (c_1-\ell_1)(c_2+\ell_2)(c_3-\ell_3) \\
      (c_1-\ell_1)\\
      (c_2+\ell_2)\\
      (c_3-\ell_3)
\end{bmatrix},\\[0.5em] &\mathbf{v}_4 \coloneq\begin{bmatrix}
      (c_1+\ell_1)(c_2+\ell_2)(c_3-\ell_3) \\
      (c_1+\ell_1)\\
      (c_2+\ell_2)\\
      (c_3-\ell_3)
\end{bmatrix}, \; \mathbf{v}_5 \coloneq\begin{bmatrix}
      (c_1-\ell_1)(c_2-\ell_2)(c_3+\ell_3) \\
      (c_1-\ell_1)\\
      (c_2-\ell_2)\\
      (c_3+\ell_3)
\end{bmatrix}, \; \mathbf{v}_6 \coloneq\begin{bmatrix}
      (c_1+\ell_1)(c_2-\ell_2)(c_3+\ell_3) \\
      (c_1+\ell_1)\\
      (c_2-\ell_2)\\
      (c_3+\ell_3)
\end{bmatrix}, \\[0.5em] &\mathbf{v}_7\coloneq\begin{bmatrix}
      (c_1-\ell_1)(c_2+\ell_2)(c_3+\ell_3) \\
      (c_1-\ell_1)\\
      (c_2+\ell_2)\\
      (c_3+\ell_3)
\end{bmatrix},\; \mathbf{v}_8 \coloneq\begin{bmatrix}
      (c_1+\ell_1)(c_2+\ell_2)(c_3+\ell_3) \\
      (c_1+\ell_1)\\
       (c_2+\ell_2)\\
      (c_3+\ell_3)
\end{bmatrix},\end{align*}}
then $\mathcal P_H^3(\mathbf{c},\Bell) = \conv\left \{\mathbf{v}_1,\mathbf{v}_2,\dots,\mathbf{v}_8 \right \}$.

Observe that for any variable $x_i$ with $i\in \{1,2,3\}$, we can partition the vertices of $\mathcal P_H^3(\mathbf{c},\Bell)$ into two sets: four vertices lying in the hyperplane $x_i=c_i-\ell_i$, and four vertices lying in the hyperplane $x_i=c_i+\ell_i$. Without loss of generality, we choose $i=3$ and denote the convex hulls of these resulting vertex sets by $Q$ and $R$, respectively. See Figure \ref{fig:cube} for an illustration.  Formally,
$$Q\coloneq\conv\{\mathbf{v}_1,\mathbf{v}_2,\mathbf{v}_3,\mathbf{v}_4\} \quad \text{and} \quad R \coloneq\conv\{\mathbf{v}_5,\mathbf{v}_6,\mathbf{v}_7,\mathbf{v}_8\}.$$

\begin{figure}
    \centering
    \includegraphics[width=0.44\linewidth]{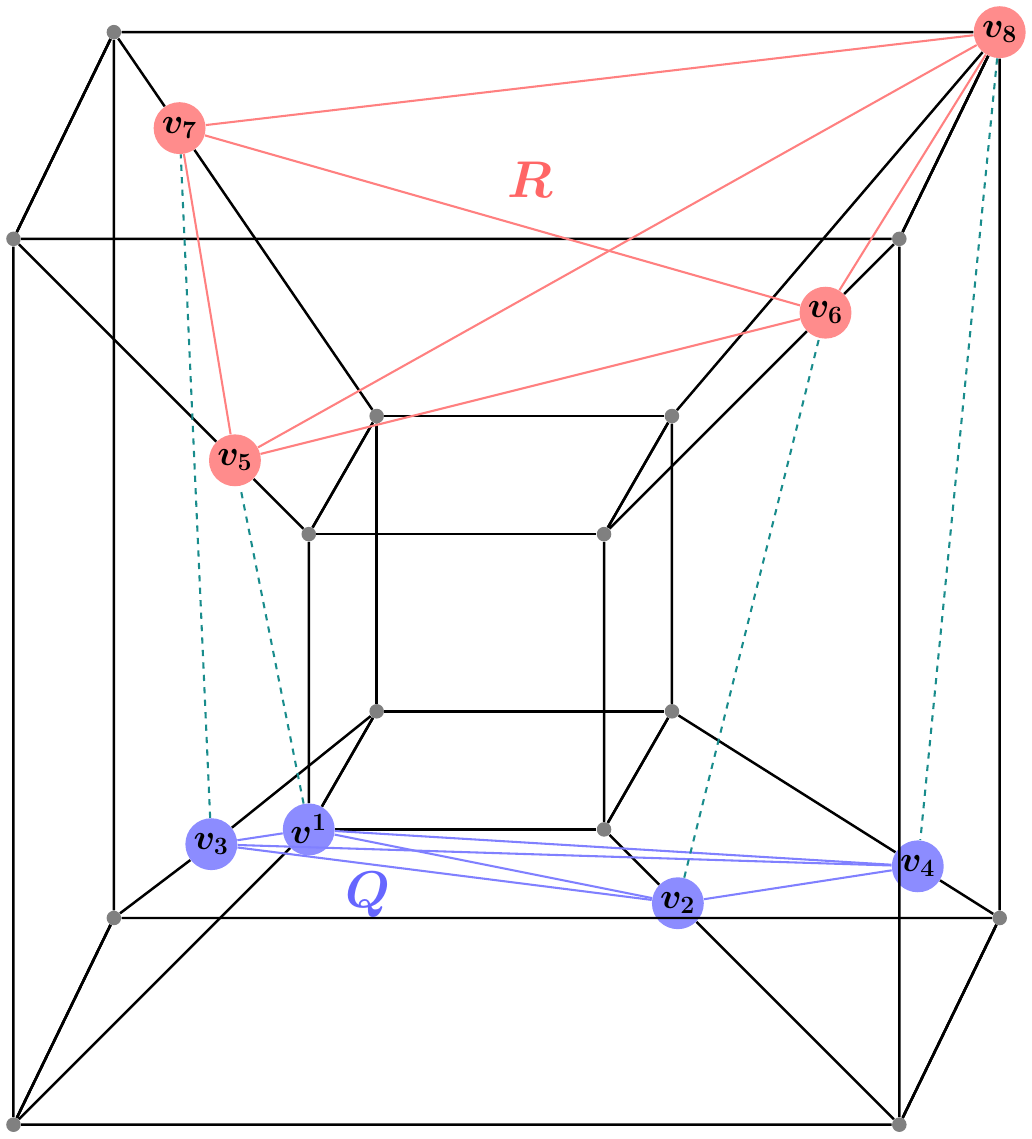}
    \caption{Adapted from \cite{SpeakmanAverkov18}.  Schlegel diagram of $\mathcal{P}_H^3(\mathbf{c}, \mathbf{\ell})$, showing its embedding in the 4-dimensional hypercube.  Note that $\mathcal{P}_H^3(\mathbf{c}, \mathbf{\ell}) = \conv{(Q\cup R)}$.}
    \label{fig:cube}
\end{figure}

By viewing $\mathcal{P}_H^3(\mathbf{c,\Bell})$ as the convex hull of the union of $Q$ and $R$, that is, two ($n-1=3$)-dimensional polytopes lying in parallel hyperplanes, its volume can be expressed as an integral over $x_3$ from $c_3-\ell_3$ to $c_3+\ell_3$:
\[
\vol{4}{\mathcal{P}_H^3(\mathbf{c,\Bell})} = \int_{x_3=c_3-\ell_3}^{x_3=c_3+\ell_3} \vol{3}{\frac{c_3+\ell_3 - x_3}{(c_3+\ell_3) - (c_3-\ell_3)} Q + \frac{x_3 - (c_3-\ell_3)}{(c_3+\ell_3) - (c_3-\ell_3)} R}\ dx_3.
\]
Note that at $x_3 = c_3-\ell_3$, the argument inside the volume function reduces to $Q$, while at $x_3 = c_3+\ell_3$, it reduces to $R$.  At intermediate values, it is a Minkowski sum of scaled copies of $Q$ and $R$.

Applying Theorems \ref{thm:mixedvol} and \ref{thm:mixedvol2} with $n = 3$, we obtain
\begin{align}
&\vol{4}{\mathcal{P}_H^3(\mathbf{c,\Bell})} = \int_{x_3=c_3-\ell_3}^{x_3=c_3+\ell_3}\vol{3}{\frac{(c_3+\ell_3) - x_3}{2\ell_3} Q + \frac{x_3 - (c_3-\ell_3)}{2\ell_3} R}\ dx_3 \notag\\[1.5em]
&\quad   = \int_{c_3-\ell_3}^{c_3+\ell_3}\left (\frac{c_3+\ell_3 - x_3}{2\ell_3}\right )^3 V(Q,Q,Q) 
+ 3\left (\frac{c_3+\ell_3 - x_3}{2\ell_3}\right)^2\left (\frac{x_3 - c_3+\ell_3}{2\ell_3}\right ) V(Q, Q, R)\notag\\[0.5em]
&\qquad \qquad \;\;  + 3\left (\frac{c_3+\ell_3 - x_3}{2\ell_3}\right )\left (\frac{x_3 - c_3+\ell_3}{2\ell_3}\right )^2 V(Q, R, R) + \left(\frac{x_3 - c_3+\ell_3}{2\ell_3}\right)^3 V(R,R,R) \ dx_3 \notag\\[1.5em]
&\quad \;\; = \frac{1}{(2\ell_3)^3}\int_{c_3-\ell_3}^{c_3+\ell_3}(c_3+\ell_3-x_3)^3\vol{3}{Q}+3(c_3+\ell_3-x_3)^2(x_3-c_3+\ell_3)V(Q,Q,R)\notag\\[0.5em]
&\qquad \qquad \;\; +3(c_3+\ell_3-x_3)(x_3-c_3+\ell_3)^2V(Q,R,R)+(x_3-c_3+\ell_3)^3\vol{3}{R}\ dx_3.\label{eq:integral}
\end{align}

Therefore, to obtain the volume of $\mathcal{P}_H^3(\mathbf{c,\Bell})$, we compute the volumes of $Q$ and $R$, and the mixed volumes $V(Q,Q,R)$ and $V(Q,R,R)$, then evaluate the integral in Equation \eqref{eq:integral}.  This is exactly the approach taken by \cite{SpeakmanAverkov18} to obtain the volume formula for the case of nonnegative bounds; however, they express \eqref{eq:integral} in terms of the upper and lower variable bounds. Under their assumptions, no case analysis was required when computing the volumes and mixed volumes, ultimately yielding a single closed-form volume formula for the convex hull polytope.  In contrast, removing the nonnegativity assumption necessitates extensive case analysis, even under the ordering \eqref{eq:omegaprime}. 

\section{Preliminary results}
\label{sec:prelim}
This section establishes the assumptions made (without loss of generality) and the preliminary results required for the proof of our main theorem.

\subsection{Assumptions}
\label{subsec:assumptions}

To reduce the case analysis required, we begin with the following symmetry observation.
\begin{lemma}\label{force_positivity}
    For $i=1,2,3$, let $\tilde{c}_i=|c_i|$.  Then 
    $$\vol{4}{\mathcal{P}_H^3(\mathbf{c},\Bell)} = \vol{4}{\mathcal{P}_H^3(\tilde{\mathbf{c}},\Bell)}.$$
\end{lemma}
\begin{proof}
For each $i \in \{1, 2, 3\}$, let $\sigma_i = \mathrm{sign}(c_i) \in \{-1, +1\}$, and consider the linear map $\phi: \mathbb{R}^4 \to \mathbb{R}^4$ defined by
$$
\phi(f, x_1, x_2, x_3) \coloneq (\sigma_1 \sigma_2 \sigma_3 \cdot f, \, \sigma_1 x_1, \, \sigma_2 x_2, \, \sigma_3 x_3),
$$
with matrix representation
$$D_\phi = \begin{bmatrix}
\sigma_1\sigma_2\sigma_3 & 0 & 0 & 0 \\
0 & \sigma_1 & 0 & 0 \\
0 & 0 & \sigma_2 & 0 \\
0 & 0 & 0 & \sigma_3
\end{bmatrix}.$$
Note that because each $\sigma_i \in \{-1, +1\}$, we have $\det(D_\phi) = \sigma_1^2 \sigma_2^2 \sigma_3^2 = 1$, thus the mapping $\phi$ preserves volume.  Moreover, we claim $\phi$ is a bijection from $\mathcal{P}^3_H(\mathbf{c}, \boldsymbol{\ell})$ to $\mathcal{P}^3_H(\tilde{\mathbf{c}}, \boldsymbol{\ell})$. 

To see this, let $(f, x_1, x_2, x_3) \in \mathcal{P}^3_H(\mathbf{c}, \boldsymbol{\ell})$, so that $f = x_1 x_2 x_3$ and 
$x_i \in [c_i - \ell_i, c_i + \ell_i]$ for $i=1,2,3$. Setting $(\tilde{f}, \tilde{x}_1, \tilde{x}_2, \tilde{x}_3) = \phi(f, x_1, x_2, x_3)$, we verify:
\begin{enumerate}[label=(\roman*)]
    \item The constraint $\tilde{f} = \tilde{x}_1 \tilde{x}_2 \tilde{x}_3$ holds:
    $$ \tilde{f} = \sigma_1\sigma_2\sigma_3 \cdot f = \sigma_1\sigma_2\sigma_3(x_1 x_2 x_3) = (\sigma_1 x_1)(\sigma_2 x_2)(\sigma_3 x_3) = \tilde{x}_1 \tilde{x}_2 \tilde{x}_3.$$    
    \item For $i=1,2,3$, we have $\tilde{x}_i = \sigma_i x_i \in [|c_i| - \ell_i, |c_i| + \ell_i]$:
    \begin{itemize}   
    \item If $c_i \geq 0$, then $\sigma_i = 1$ and $\tilde{x}_i = x_i \in [c_i - \ell_i, c_i + \ell_i] = [|c_i| - \ell_i, |c_i| + \ell_i]$. 
   \item  If $c_i < 0$, then $\sigma_i = -1$ and since $x_i \in [c_i - \ell_i, c_i + \ell_i]$, we have
    \[
    \tilde{x}_i = -x_i \in [-(c_i + \ell_i), -(c_i - \ell_i)] = [-c_i - \ell_i, -c_i + \ell_i] = [|c_i| - \ell_i, |c_i| + \ell_i].
    \]
    \end{itemize}
\end{enumerate}
Therefore, $\phi$ maps $\mathcal{P}^3_H(\mathbf{c}, \boldsymbol{\ell})$ into $\mathcal{P}^3_H(\tilde{\mathbf{c}}, \boldsymbol{\ell})$, and because it is invertible with $\phi^{-1} = \phi$, it is a bijection. Thus, the mapping $\phi$ is a volume-preserving bijection and
$$\vol{4}{\mathcal{P}_H^3(\mathbf{c},\Bell)} = \vol{4}{\mathcal{P}_H^3(\tilde{\mathbf{c}},\Bell)}$$
as required. 
\end{proof}
\vspace{4mm}

As a consequence of Lemma~\ref{force_positivity}, for the remainder of the paper, we assume $c_i\geq 0$ without loss of generality for $i=1,2,3$.

\label{rem:continuity}
The proofs in the remainder of this paper additionally assume $c_3 \neq \ell_3$ (equivalently, $a_3 \neq 0$). At the boundary $c_3 = \ell_3$, the polytope $Q$ degenerates to a rectangle in a 2-dimensional subspace and the derivation of $\mathcal{U}(Q)$ via Lemma~\ref{lemma:outerfacetnormals} does not apply. However, by continuity of both the volume of the convex hull and our derived formulae in $(\mathbf{c},\boldsymbol{\ell})$, the formulae remain correct at $c_3 = \ell_3$. We therefore state the main result (Theorem~\ref{bigtheorem}) under the non-strict ordering, with the degenerate case included by continuity. The precise continuity argument required is detailed in \cite[Section 4]{SpeakmanLee17}.

\subsection{Computing the volumes of ${Q}$ and ${R}$}
To compute the (3-dimensional) volumes of $Q$ and $R$, we simply need to observe that they are both 3-dimensional simplices, in other words, tetrahedra.  The volume of each may be computed using the standard determinant formula for simplices.
\begin{lemma}\label{volQ}
The volume of $Q$ is given by
    $$\vol{3}{Q}=\begin{cases}
    \frac{8}{3}\ell_1^2\ell_2^2\left(c_3-\ell_3\right), &\text{if } \frac{c_3}{\ell_3} \mathcolor{teal}{>} 1, \\
    \frac{8}{3}\ell_1^2\ell_2^2\left(\ell_3-c_3\right), &\text{if } \frac{c_3}{\ell_3} < 1. 
    \end{cases}$$
\end{lemma}
\vspace{4mm}

\begin{lemma}\label{volR}
The volume of $R$ is given by
    $$\vol{3}{R}=\frac{8}{3}\ell_1^2\ell_2^2(c_3+\ell_3).$$
\end{lemma}
\vspace{4mm}

Note that due to our nonnegativity assumption on $\mathbf{c}$, we have $c_3+\ell_3 > 0$, and we do not need to consider separate cases when computing the volume of $R$.

\subsection{Computing the mixed volumes ${V(Q,Q,R)}$ and ${V(Q,R,R)}$}
By Theorem \ref{thm:polytopemixedvol}, we have
$$V(Q,Q,R)=\frac{1}{3}\sum_{\mathbf{u}\in\mathcal{U}(Q)}h_R(\mathbf{u}) \qquad \text{and} \qquad V(Q,R,R)=\frac{1}{3}\sum_{\mathbf{u}\in\mathcal{U}(R)}h_Q(\mathbf{u}),$$
where $\mathcal{U}(Q)$ and $\mathcal{U}(R)$ denote the sets of outer facet normals of $Q$ and $R$ (Definition~\ref{def:outerfacetnromals}), and $h_Q(\cdot)$ and $h_R(\cdot)$ denote the support functions of $Q$ and $R$ (Definition~\ref{def:support function}).
\subsubsection{Deriving the outer facet normals}
\label{subsec:outerfacetnormals}
\phantom{new line}

Define the indicator $$I=\begin{cases}
    1 &\text{ if $\frac{c_3}{\ell_3}<1$},\\
    0 &\text{ if $\frac{c_3}{\ell_3}\mathcolor{teal}{>}1$}.
\end{cases}$$ 
Applying Lemma \ref{lemma:outerfacetnormals} and embedding into $\mathbb{R}^4$ with a zero component in the $x_3$ direction, we obtain 
\begin{align*}
    \mathcal{U}(Q)
&= (-1)^{I}2\ell_1\ell_2\bigggl\{
{ \small \begin{bmatrix}
      1 \\
      -(c_2-\ell_2)(c_3-\ell_3)\\
      -(c_1+\ell_1)(c_3-\ell_3)\\
      0
\end{bmatrix}},
{\small\begin{bmatrix}
      1 \\
      -(c_2+\ell_2)(c_3-\ell_3)\\
      -(c_1-\ell_1)(c_3-\ell_3)\\
      0
\end{bmatrix}},
{\small\begin{bmatrix}
      -1 \\
      (c_2+\ell_2)(c_3-\ell_3)\\
      (c_1+\ell_1)(c_3-\ell_3)\\
      0
\end{bmatrix}}, \\[0em]
& \hspace{90mm}{\small\begin{bmatrix}
      -1\\
      (c_2-\ell_2)(c_3-\ell_3)\\
      (c_1-\ell_1)(c_3-\ell_3)\\
      0
\end{bmatrix}}
\bigggr \}\\
&\coloneq\begin{cases}2\ell_1\ell_2\left\{\mathbf{u}_1,\mathbf{u}_2,\mathbf{u}_3,\mathbf{u}_4\right\}, & \frac{c_3}{\ell_3}\mathcolor{teal}{>}1, \\2\ell_1\ell_2\left\{-\mathbf{u}_1,-\mathbf{u}_2,-\mathbf{u}_3,-\mathbf{u}_4\right\}, & \frac{c_3}{\ell_3}< 1, \end{cases}
\end{align*}
and

\begin{align*}
 \mathcal{U}(R) 
&= 2\ell_1\ell_2\bigggl\{
{\small\begin{bmatrix}
      1 \\
      -(c_2-\ell_2)(c_3+\ell_3)\\
      -(c_1+\ell_1)(c_3+\ell_3)\\
      0
\end{bmatrix}},
{\small\begin{bmatrix}
      1 \\
      -(c_2+\ell_2)(c_3+\ell_3)\\
      -(c_1-\ell_1)(c_3+\ell_3)\\
      0
\end{bmatrix}},
{\small\begin{bmatrix}
      -1 \\
      (c_2+\ell_2)(c_3+\ell_3)\\
      (c_1+\ell_1)(c_3+\ell_3)\\
      0
\end{bmatrix}}, \\[0em]
& \hspace{80mm}{\small\begin{bmatrix}
      -1\\
      (c_2-\ell_2)(c_3+\ell_3)\\
      (c_1-\ell_1)(c_3+\ell_3)\\
      0
\end{bmatrix}}
\bigggr\}\\
&\coloneq2\ell_1\ell_2\left\{\mathbf{u}_5,\mathbf{u}_6,\mathbf{u}_7,\mathbf{u}_8\right\}.
\end{align*}
Observe that when $\frac{c_3}{\ell_3}<1$, 
we must relabel the extreme points of $Q$ in order to satisfy the assumptions of Lemma \ref{lemma:outerfacetnormals}. This results in the two distinct cases for  $\mathcal{U}(Q)$.

\subsubsection{Evaluating the support functions}
To compute the necessary mixed volumes, we must evaluate $h_Q(\cdot)$ at each point in $\mathcal{U}(R)$ and evaluate $h_R(\cdot)$ at each point in $\mathcal{U}(Q)$.  

Recall by Definition \ref{def:support function},
$$h_Q (\mathbf{u}) = \max_{\mathbf{x}\in Q} \mathbf{x}^T\mathbf{u} \qquad \text{and} \qquad h_R(\mathbf{u}) = \max_{\mathbf{x}\in R} \mathbf{x}^T\mathbf{u}.$$
Moreover, because $Q$
and $R$ are tetrahedra (embedded in $\R^4$), to compute $h_R(\mathbf{u})$ or $h_Q(\mathbf{u})$ for any vector $\mathbf{u}$, we need only compute $\mathbf{x}^T\mathbf{u}$ for each of the four extreme points, $\mathbf{x}$, and take the maximum of these values. Note that because the signs of the outer facet normals in $\mathcal{U}(Q)$ depend on the indicator $I$, the inner products for the case $\frac{c_3}{\ell_3}\mathcolor{teal}{>}1$ and $\frac{c_3}{\ell_3}<1$ differ by a factor of $-1$. This leads to further case analysis when evaluating the mixed volumes.

\subsubsection{Putting it together}
Throughout the proofs of Propositions~\ref{prop:V(Q,Q,R)} and~\ref{prop:V(Q,R,R)}, we refer to Lemmas~\ref{z_1}--\ref{z_8} whose statements and proofs appear in the \hyperlink{appendix}{appendix}.

\begin{proposition}\label{prop:V(Q,Q,R)}
Assume $c_i \geq 0$ for $i = 1,2,3$, and the ordering \eqref{eq:omegaprime} holds.  Then
\begin{equation*}
       V(Q,Q,R) =
        \begin{cases} 
       \frac{8}{3} \ell _1 \ell _2 \left(c_3 \ell _1 \ell _2+\ell _3 \left(2 c_2 \ell _1+c_1 \ell _2\right)\right),&\text{if }\frac{c_1}{\ell_1}\geq1,\\[0.5em]
       \frac{8}{3} \ell _1^2 \ell _2 \left(2 c_2 \ell _3+\ell _2 \left(c_3+\ell _3\right)\right),& \text{if }
        \frac{c_1}{\ell_1}<1, \frac{c_1}{\ell_1} + \frac{c_2}{\ell_2} + \frac{c_3}{\ell_3}\geq 1\\&\hspace{5mm}\text{ and } \frac{c_2}{\ell_2} +  \frac{c_3}{\ell_3} \geq 1 + \frac{c_1}{\ell_1},
        \\[0.5em]
         \frac{8}{3} \ell_1\ell_2\ell _3 \left(c_2 \ell _1+\ell _2 \left(c_1+2 \ell _1\right)\right),&\text{if }
        \frac{c_1}{\ell_1}<1, \frac{c_1}{\ell_1} + \frac{c_2}{\ell_2} + \frac{c_3}{\ell_3}\geq 1\\&\hspace{5mm}\text{ and } \frac{c_2}{\ell_2} +  \frac{c_3}{\ell_3} < 1 + \frac{c_1}{\ell_1},\\[0.5em]
        \frac{8}{3} \ell _1^2 \ell _2^2 \left(3 \ell_3-c_3\right),&\text{if } \frac{c_1}{\ell_1} + \frac{c_2}{\ell_2} + \frac{c_3}{\ell_3}<1.
        \end{cases}   
    \end{equation*}
\end{proposition}
\begin{proof}
Let $E(R) = \{\mathbf{v}_5, \mathbf{v}_6, \mathbf{v}_7, \mathbf{v}_8\}$ denote the vertices of $R$. From Section \ref{subsec:outerfacetnormals}, the set of outer facet normals is $$\mathcal{U}(Q) = \begin{cases} 2\ell_1\ell_2\{\mathbf{u}_1, \mathbf{u}_2, \mathbf{u}_3, \mathbf{u}_4\}&\frac{c_3}{\ell_3} \mathcolor{teal}{>} 1 \\ 2\ell_1\ell_2\{-\mathbf{u}_1, -\mathbf{u}_2, -\mathbf{u}_3, -\mathbf{u}_4\} & \frac{c_3}{\ell_3} <1. \end{cases} $$  Therefore, by Theorem \ref{thm:polytopemixedvol} and Definition \ref{def:support function}, we have
\begin{equation*}\label{eq:V_QQR_formula}
V(Q,Q,R) = \begin{cases} \displaystyle \frac{1}{3}\displaystyle\sum_{i= 1}^4 \displaystyle \max_{\mathbf{v} \in E(R)}\mathbf{v}^T\mathbf{u}_i , & \frac{c_3}{\ell_3} \mathcolor{teal}{>} 1, \\ \displaystyle\frac{1}{3}\displaystyle\sum_{i= 1}^4 \displaystyle\max_{\mathbf{v} \in E(R)}-\mathbf{v}^T\mathbf{u}_i , & \frac{c_3}{\ell_3} < 1.\end{cases}
\end{equation*}

For each outer facet normal $\mathbf{u}_i$, with $i \in \{1,2,3,4\}$, define:
\begin{itemize}
    \item $z^+_i = \max\{\mathbf{v}^T_5\mathbf{u}_i, \mathbf{v}^T_6\mathbf{u}_i, \mathbf{v}^T_7\mathbf{u}_i, \mathbf{v}^T_8\mathbf{u}_i\}$ (the support function evaluated at $\mathbf{u}_i$), and
      \item $z^-_i = \max\{-\mathbf{v}^T_5\mathbf{u}_i, -\mathbf{v}^T_6\mathbf{u}_i, -\mathbf{v}^T_7\mathbf{u}_i, -\mathbf{v}^T_8\mathbf{u}_i\}$ (the support function evaluated at $-\mathbf{u}_i$).
\end{itemize}
Then we have 
\begin{equation}\label{eq:V_QQR_z_sum}
V(Q,Q,R) = \begin{cases} \frac{1}{3}\cdot 2\ell_1\ell_2\cdot(z^+_1 + z^+_2 + z^+_3 + z^+_4) & \frac{c_3}{\ell_3} \mathcolor{teal}{>}1\\ \frac{1}{3}\cdot 2\ell_1\ell_2\cdot(z^-_1 + z^-_2 + z^-_3 + z^-_4) &  \frac{c_3}{\ell_3} <1.\end{cases}
\end{equation}
The values $z^+_i$ and $z^-_i$ depend on the parameter ratios and are determined by Lemmas \ref{z_1}--\ref{z_4} in the \hyperlink{appendix}{appendix}. This leads to four cases.

\medskip
\noindent\textbf{Case 1: $\frac{c_1}{\ell_1} \geq 1$.} 

By the ordering  \eqref{eq:omegaprime}, $\frac{c_1}{\ell_1} \geq 1$ implies $\frac{c_3}{\ell_3} \mathcolor{teal}{>} 1$.

Then by Lemmas \ref{z_1}--\ref{z_4}, after substitution and simplification, we have 
\begin{equation*}
z^+_1 + z^+_2 + z^+_3 + z^+_4 = \mathbf{v}^T_8\mathbf{u}_1 + \mathbf{v}^T_8\mathbf{u}_2+ \mathbf{v}^T_7\mathbf{u}_3 + \mathbf{v}^T_5\mathbf{u}_4 =4(c_3\ell_1\ell_2 + \ell_3(2c_2\ell_1 + c_1\ell_2)),
\end{equation*}
and substituting into \eqref{eq:V_QQR_z_sum} gives the first formula.

\medskip
\noindent\textbf{Case 2: $\frac{c_1}{\ell_1} < 1$ and $\frac{c_1}{\ell_1} + \frac{c_2}{\ell_2} + \frac{c_3}{\ell_3} \geq 1$ and $\frac{c_2}{\ell_2} +  \frac{c_3}{\ell_3} \geq 1 + \frac{c_1}{\ell_1}$.}

If $\frac{c_3}{\ell_3} \mathcolor{teal}{>} 1$, by Lemmas \ref{z_1}--\ref{z_4}, after substitution and simplification, we have   
\begin{equation*}
z^+_1 + z^+_2 + z^+_3 + z^+_4 =  \mathbf{v}^T_8\mathbf{u}_1 + \mathbf{v}^T_8\mathbf{u}_2 + \mathbf{v}^T_7\mathbf{u}_3 + \mathbf{v}^T_7\mathbf{u}_4 = 4\ell_1(2c_2\ell_3 + \ell_2(c_3+\ell_3)),
\end{equation*}

If $\frac{c_3}{\ell_3} < 1$, by Lemmas \ref{z_1}--\ref{z_4}, after substitution and simplification, we have   
\begin{equation*}
z^-_1 + z^-_2 + z^-_3 + z^-_4 =  -\mathbf{v}^T_7\mathbf{u}_1 - \mathbf{v}^T_7\mathbf{u}_2 - \mathbf{v}^T_8\mathbf{u}_3 - \mathbf{v}^T_8\mathbf{u}_4 = 4\ell_1(2c_2\ell_3 + \ell_2(c_3+\ell_3)).
\end{equation*}

In both subcases, substituting into \eqref{eq:V_QQR_z_sum} gives the second formula.

\medskip
\noindent\textbf{Case 3:  $\frac{c_1}{\ell_1} < 1$ and $\frac{c_1}{\ell_1} + \frac{c_2}{\ell_2} + \frac{c_3}{\ell_3} \geq 1$ and $\frac{c_2}{\ell_2} +  \frac{c_3}{\ell_3} < 1 + \frac{c_1}{\ell_1}$.}  

This case implies $\frac{c_3}{\ell_3} < 1$.

By Lemmas \ref{z_1}--\ref{z_4}, 
\begin{equation*}
z^-_1 + z^-_2 + z^-_3 + z^-_4 =  -\mathbf{v}^T_6\mathbf{u}_1 - \mathbf{v}^T_7\mathbf{u}_2 - \mathbf{v}^T_8\mathbf{u}_3 - \mathbf{v}^T_8\mathbf{u}_4 = 4 \ell _3 \left(c_2 \ell _1+\ell _2 \left(c_1+2 \ell _1\right)\right).
\end{equation*}
Substituting into \eqref{eq:V_QQR_z_sum} gives the third formula. 

\medskip
\noindent\textbf{Case 4: $\frac{c_1}{\ell_1} + \frac{c_2}{\ell_2} + \frac{c_3}{\ell_3} < 1$ }  

This case implies $\frac{c_3}{\ell_3} < 1$ and $\frac{c_2}{\ell_2} +  \frac{c_3}{\ell_3} < 1 + \frac{c_1}{\ell_1}$.

By Lemmas \ref{z_1}--\ref{z_4}, 
\begin{equation*}
z^-_1 + z^-_2 + z^-_3 + z^-_4 =  -\mathbf{v}^T_6\mathbf{u}_1 - \mathbf{v}^T_7\mathbf{u}_2 - \mathbf{v}^T_8\mathbf{u}_3 - \mathbf{v}^T_5\mathbf{u}_4 = 4\ell_1\ell_2(3\ell_3-c_3).
\end{equation*}
Substituting into \eqref{eq:V_QQR_z_sum} gives the fourth formula.  
\end{proof}
\vspace{4mm}

\begin{proposition}\label{prop:V(Q,R,R)}
Assume $c_i \geq 0$ for $i = 1,2,3$, and the ordering \eqref{eq:omegaprime} holds.  Then
\begin{equation*}
        V(Q,R,R) = 
        \begin{cases}
        \frac{8}{3} \ell _1 \ell _2 \left(c_3 \ell _1 \ell _2+\ell _3 \left(2 c_2 \ell _1+c_1 \ell _2\right)\right), & \text{ if } \frac{c_1}{\ell_1}\geq1,\\[0.5em]
\frac{8}{3} \ell _1^2 \ell _2 \left(2 c_2 \ell _3+\ell _2 \left(c_3+\ell _3\right)\right), &  \text{ if } \frac{c_1}{\ell_1} <1 \text{ and } \frac{c_2}{\ell_2} \geq 1, \\[0.5em]     
        \frac{8}{3} \ell _1^2 \ell _2^2 \left(c_3+3 \ell _3\right), & \text{otherwise.}
        \end{cases}
    \end{equation*}
\end{proposition}

\begin{proof}
Let $E(Q) = \{\mathbf{v}_1, \mathbf{v}_2, \mathbf{v}_3, \mathbf{v}_4\}$ denote the vertices of $Q$. From Section \ref{subsec:outerfacetnormals}, the set of outer facet normals is $$\mathcal{U}(R) = 2\ell_1\ell_2\{\mathbf{u}_5, \mathbf{u}_6, \mathbf{u}_7, \mathbf{u}_8\}.$$
 Therefore, by Theorem \ref{thm:polytopemixedvol} and Definition \ref{def:support function}, we have
\begin{equation*}\label{eq:V_QRR_formula}
V(Q,R,R)  = \frac{1}{3}\sum_{\mathbf{u}_i\in\mathcal{U}(R)} \max_{\mathbf{v} \in E(Q)} \mathbf{v}^T\mathbf{u}_i.
\end{equation*}

For each outer facet normal $\mathbf{u}_i$ , with $i \in \{5,6,7,8\}$, define:

\begin{center} $z_i = \max\{\mathbf{v}^T_1\mathbf{u}_i, \mathbf{v}^T_2\mathbf{u}_i, \mathbf{v}^T_3\mathbf{u}_i, \mathbf{v}^T_4\mathbf{u}_i\}$ (the support function evaluated at $\mathbf{u}_i$). \end{center}

 Then we have 
\begin{equation}\label{eq:V_QRR_z_sum}
V(Q,R,R) = \frac{1}{3} \cdot 2\ell_1\ell_2 \cdot (z_5 + z_6 + z_7 + z_8).
\end{equation}
The values $z_i$ depend on the parameter ratios and are determined by Lemmas \ref{z_5}--\ref{z_8} in the \hyperlink{appendix}{appendix}. We consider three cases.

\medskip
\noindent\textbf{Case 1: $\frac{c_1}{\ell_1} \geq 1$.} 

By the ordering  \eqref{eq:omegaprime}, this implies $\frac{c_2}{\ell_2} \geq 1$. 
By Lemmas \ref{z_5}--\ref{z_8}, after substitution and simplification,
\begin{equation*}
z_5 + z_6 + z_7 + z_8 =  \mathbf{v}^T_1\mathbf{u}_5+\mathbf{v}^T_1\mathbf{u}_6 + \mathbf{v}^T_4\mathbf{u}_7 + \mathbf{v}^T_2\mathbf{u}_8= 4(c_3\ell_1\ell_2 + \ell_3(2c_2\ell_1 + c_1\ell_2)).
\end{equation*}
Substituting into \eqref{eq:V_QRR_z_sum} gives the first formula.

\medskip
\noindent\textbf{Case 2: $\frac{c_1}{\ell_1} < 1$ and $\frac{c_2}{\ell_2} \geq 1$.}

By Lemmas \ref{z_5}--\ref{z_8}, after substitution and simplification,
\begin{equation*}
z_5 + z_6 + z_7 + z_8 = \mathbf{v}^T_1\mathbf{u}_5+\mathbf{v}^T_3\mathbf{u}_6 + \mathbf{v}^T_4\mathbf{u}_7 + \mathbf{v}^T_2\mathbf{u}_8 = 4\ell_1(2c_2\ell_3 + \ell_2(c_3+\ell_3)).
\end{equation*}
Substituting into \eqref{eq:V_QRR_z_sum} gives the second formula.

\medskip
\noindent\textbf{Case 3: $\frac{c_1}{\ell_1} < 1$ and $\frac{c_2}{\ell_2} < 1$.}

By Lemmas \ref{z_5}--\ref{z_8},
\begin{equation*}
z_5 + z_6 + z_7 + z_8 =  \mathbf{v}^T_2\mathbf{u}_5+\mathbf{v}^T_3\mathbf{u}_6 + \mathbf{v}^T_4\mathbf{u}_7 + \mathbf{v}^T_1\mathbf{u}_8 = 4\ell_1\ell_2(c_3+3\ell_3).
\end{equation*}
Substituting into \eqref{eq:V_QRR_z_sum} gives the third formula.  
\end{proof}
\section{Main Result}
\label{sec:main}
We now combine the results of the previous sections to obtain a closed-form volume formula for $\mathcal P_H^3(\mathbf{c},\mathbf{\ell})$ on a general box domain. The formula splits into six cases, which together partition the parameter space under the ordering \eqref{eq:omegaprime}. 

\begin{theorem}\label{bigtheorem}
Assume $c_i \geq 0$ for $i = 1,2,3$, and the ordering \eqref{eq:omegaprime} holds.  Then
    $$\vol{4}{\mathcal{P}_H^3} =\begin{cases}
        \frac{8}{3}\ell_1^2\ell_2^2\ell_3^2\left(\frac{c_1}{\ell_1}+2\frac{c_2}{\ell_2}+2\frac{c_3}{\ell_3}\right),&\text{ if $\frac{c_1}{\ell_1}\geq1$},\\[0.5em]
        \frac{8}{3}\ell_1^2\ell_2^2\ell_3^2\left(2\frac{c_2}{\ell_2}+2\frac{c_3}{\ell_3} + 1\right),&\text{ if $\frac{c_1}{\ell_1}<1$ and $\frac{c_2}{\ell_2}\geq1$},\\[0.5em]
        \frac{8}{3}\ell_1^2\ell_2^2\ell_3^2\left(\frac{c_2}{\ell_2}+2\frac{c_3}{\ell_3} + 2\right),&\text{ if $\frac{c_2}{\ell_2}<1$ and $\frac{c_3}{\ell_3}\geq1$},\\[0.5em]
        \frac{8}{3}\ell_1^2\ell_2^2\ell_3^2\left(\frac{c_2}{\ell_2}+\frac{c_3}{\ell_3} + 3\right),&\text{ if $\frac{c_3}{\ell_3}<1$ and $\frac{c_1}{\ell_1}+\frac{c_2}{\ell_2}+\frac{c_3}{\ell_3}\geq 1$, and $\frac{c_2}{\ell_2}+\frac{c_3}{\ell_3}\geq 1+\frac{c_1}{\ell_1}$}, \\[0.5em]
        \frac{4}{3}\ell_1^2\ell_2^2\ell_3^2\left(\frac{c_1}{\ell_1}+\frac{c_2}{\ell_2}+\frac{c_3}{\ell_3} + 7\right),&\text{ if $\frac{c_3}{\ell_3}<1$ and $\frac{c_1}{\ell_1}+\frac{c_2}{\ell_2}+\frac{c_3}{\ell_3}\geq 1$, and $\frac{c_2}{\ell_2}+\frac{c_3}{\ell_3}< 1+\frac{c_1}{\ell_1}$},\\[0.5em]
        \frac{32}{3}\ell_1^2\ell_2^2\ell_3^2,&\text{ if $\frac{c_3}{\ell_3}<1$ and $\frac{c_1}{\ell_1}+\frac{c_2}{\ell_2}+\frac{c_3}{\ell_3} < 1$}.
    \end{cases}$$

    These six cases form a complete partition of the parameter space under the assumptions.
\end{theorem}

\begin{proof}
The case $c_3 = \ell_3$ follows by continuity (see Section~\ref{subsec:assumptions}).  Therefore, it suffices to prove the result under the strict assumption $c_3 \neq \ell_3$.  

The result follows from substituting the volumes of $Q$ and $R$ (Lemmas~\ref{volQ} and~\ref{volR}), and the mixed volumes $V(Q,Q,R)$ and $V(Q,R,R)$ (Propositions~\ref{prop:V(Q,Q,R)} and~\ref{prop:V(Q,R,R)}) into Equation~\eqref{eq:integral}, evaluating the integral, and simplifying.  The six cases in the theorem correspond to different case combinations that arise when computing these volumes and mixed volumes.  For a summary, see Table \ref{tab:theorem6-cases}. Recall that there is no case analysis for the volume of $R$, hence we use $$\vol{3}{R}=\frac{8}{3}\ell_1^2\ell_2^2(c_3+\ell_3),$$ in every case of the theorem, and Lemma~\ref{volR} is omitted from the table.  
\end{proof}

\begin{table}[h]
\centering
\begin{tabularx}{\textwidth}{clCCC}
\hline
\multirowcell{2}{\\ Case} & \multirowcell{2}{\\ Conditions} & \multirowcell{2}{Proposition~\ref{prop:V(Q,Q,R)}\\ $V(Q,Q,R)$} & \multirowcell{2}{Proposition~\ref{prop:V(Q,R,R)}\\ $V(Q,R,R)$} & \multirowcell{2}{\\ Lemma~\ref{volQ}\\ $\vol{3}{Q}$}   \\
 &  & & & \\\hline
1 & $\frac{c_1}{\ell_1} \geq 1$ & Case 1 & Case 1 & Case 1 \\[5pt]
2 & $\frac{c_1}{\ell_1} < 1$, $\frac{c_2}{\ell_2} \geq 1$ & Case 2 & Case 2 &Case 1 \\[5pt]
3 & $\frac{c_2}{\ell_2} < 1$, $\frac{c_3}{\ell_3} \mathcolor{teal}{>} 1$ & Case 2 & Case 3 & Case 1 \\[5pt]
4 & $\frac{c_3}{\ell_3} < 1$, $\frac{c_1}{\ell_1}+\frac{c_2}{\ell_2}+\frac{c_3}{\ell_3} \geq 1$, $\frac{c_2}{\ell_2} + \frac{c_3}{\ell_3} \geq 1 + \frac{c_1}{\ell_1}$ & Case 2 & Case 3 & Case 2 \\[5pt]
5 & $\frac{c_3}{\ell_3} < 1$, $\frac{c_1}{\ell_1}+\frac{c_2}{\ell_2}+\frac{c_3}{\ell_3} \geq 1$, $\frac{c_2}{\ell_2} + \frac{c_3}{\ell_3} < 1 + \frac{c_1}{\ell_1}$ & Case 3 & Case 3 & Case 2 \\[5pt]
6 & $\frac{c_3}{\ell_3} < 1$, $\frac{c_1}{\ell_1}+\frac{c_2}{\ell_2}+\frac{c_3}{\ell_3} < 1$ & Case 4 & Case 3 & Case 2\\ \hline
\end{tabularx}

\vspace{4mm}

\caption{Case mappings for the proof of Theorem~\ref{bigtheorem}. Each row indicates which cases from Proposition~\ref{prop:V(Q,Q,R)}, Proposition~\ref{prop:V(Q,R,R)}, and Lemma~\ref{volQ} are used to evaluate the integral in Equation \eqref{eq:integral} for the given case of Theorem~\ref{bigtheorem}.}
\label{tab:theorem6-cases}
\end{table}

\subsection{Example}
Consider the polytope

$$P\coloneq\text{conv}\{(f,x_1,x_2,x_3)\in\mathbb R^4\ :\ f=x_1x_2x_3,\ x_1\in[3,7],\ x_2 \in [-2,4],\ x_3 \in [-3,-1]\}.$$

To find the volume of $P$, we first compute the centers and half-lengths of each interval.  Applying Lemma~\ref{force_positivity}, we take $c_1=5, c_2=1, c_3=2$, and $\ell_1 = 2, \ell_2=3, \ell_3 = 1$.

To ensure the correct labeling, we calculate: $$\frac{c_1}{\ell_1} = \frac{5}{2}, \qquad \frac{c_2}{\ell_2} = \frac{1}{3}, \qquad \text{and} \qquad \frac{c_3}{\ell_3} = \frac{2}{1},$$ and observe that we must permute the indices to satisfy \eqref{eq:omegaprime}:

$$(1,2,3) \mapsto (3,1,2).$$

Therefore, we define the polytope 
$$\tilde{P}\coloneq \text{conv}\{(f,x_1,x_2,x_3)\in\mathbb R^4\ :\ f=x_1x_2x_3,\ x_1\in[-2,4],\ x_2 \in [1,3],\ x_3 \in [3,7]\},$$

and note that $\vol{4}{P} = \vol{4}{\tilde{P}}$ by Lemma~\ref{force_positivity}.

The plots in Figure \ref{fig:examplegraphics} visualize the 4-dimensional polytope $\tilde{P}$ by displaying five representative examples (at $x_3 \in \{3,4,5,6,7\}$) from the continuum of 3-dimensional cross-sections of $\tilde{P}$ parallel to the $x_3=0$ plane.   Note that when $x_3=3$ we have exactly the tetrahedron $Q$, and when $x_3=7$, we have $R$.  ``Slices'' of $\tilde{P}$, such as these examples, correspond to  the argument inside the volume function in Equation \eqref{eq:integral}.  Specifically, for a fixed $x_3\in [3,7]$, the corresponding cross-section may be described as   %
$$\frac{c_3+\ell_3 - x_3}{(c_3+\ell_3) - (c_3-\ell_3)} Q + \frac{x_3 - (c_3-\ell_3)}{(c_3+\ell_3) - (c_3-\ell_3)} R \;\; = \;\; \frac{7 - x_3}{4} Q + \frac{x_3 - 3}{4} R,$$
and Theorem \ref{bigtheorem} provides the formula for \[
\vol{4}{\tilde{P}} = \int_{x_3=3}^{x_3=7} \vol{3}{\frac{7 - x_3}{4} Q + \frac{x_3 - 3}{4} R} dx_3.
\]
Taking into account the appropriate labeling, observe that we fall into Case 2 of Theorem \ref{bigtheorem}. We can now compute the volume of $\tilde{P}$, and thus the volume of $P$, as follows:
\begin{align*}
\vol{4}{P} &= \vol{4}{\tilde{P}}\\
&=\frac{8}{3}\ell_1^2\ell_2^2\ell_3^2\left(2\frac{c_2}{\ell_2}+2\frac{c_3}{\ell_3} + 1\right) \\
&=\frac{8}{3}(3)^2(1)^2(2)^2\left(2 \cdot \frac{2}{1}+2 \cdot \frac{5}{2} + 1\right) \\
&=\frac{8}{3} \;\times\; 36 \;\times\; 10 \\
&=960.
\end{align*}
\noindent The reader may verify this result numerically using software.

\begin{figure}[h!]
\centering 
\begin{subfigure}{0.32\textwidth}
\centering 
    \includegraphics[width=\textwidth]{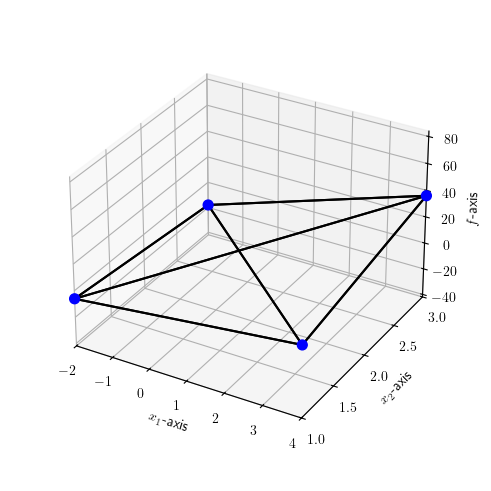}
    \caption{Cross-section when \\ $x_3=3$}
\end{subfigure} 
\begin{subfigure}{0.32\textwidth}
\centering 
    \includegraphics[width=\linewidth]{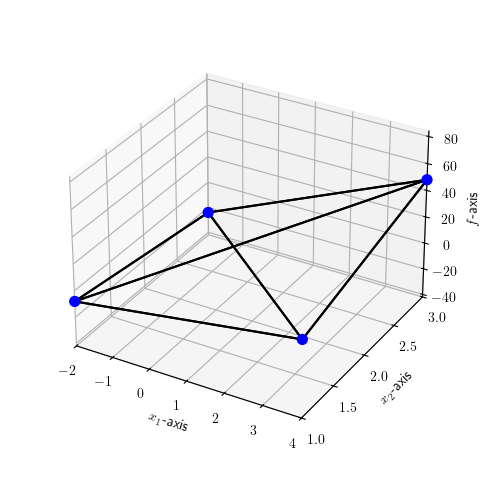}
    \caption{Cross-section when \\$x_3=4$}
\end{subfigure} 
\begin{subfigure}{0.32\textwidth}
\centering 
    \includegraphics[width=\linewidth]{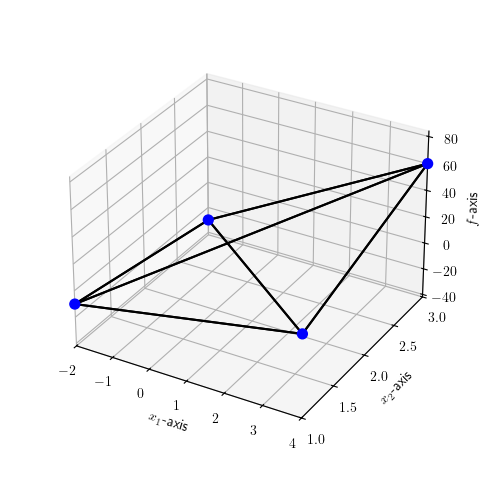}
    \caption{Cross-section when \\$x_3=5$}
\end{subfigure} 
\begin{subfigure}{0.32\textwidth}
\centering 
    \includegraphics[width=\linewidth]{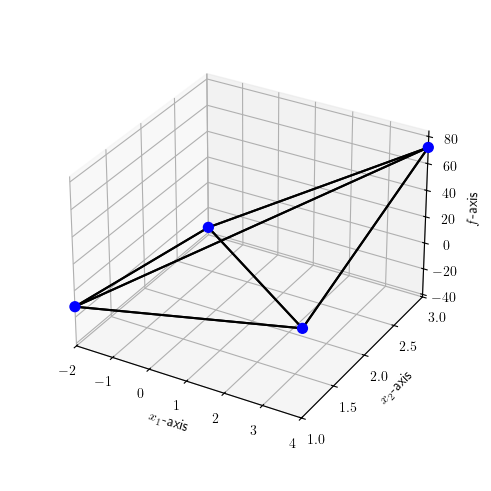}
    \caption{Cross-section when \\$x_3=6$}
\end{subfigure}
\begin{subfigure}{0.32\textwidth}
\centering 
    \includegraphics[width=\linewidth]{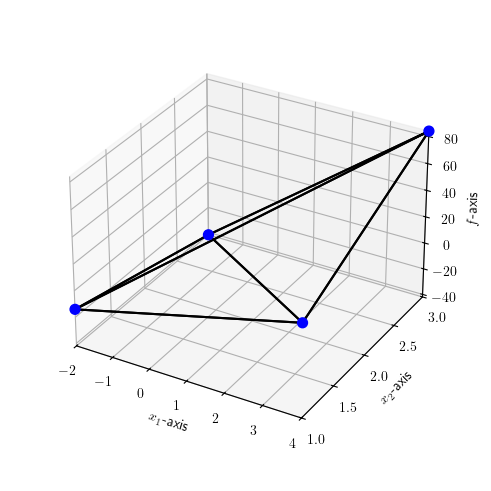}
    \caption{Cross-section when \\$x_3=7$}
\end{subfigure}
\caption{Cross-sections of $\tilde{P}$ for $x_3 \in \{3,4,5,6,7\}$.  Note that when $x_3=3$, the corresponding cross-section is $Q$, and when $x_3=7$, the corresponding cross-section is $R$.  
}
\label{fig:examplegraphics}
\end{figure}

\subsection{Numerical comparison with double McCormick relaxations}\label{sec:mccormick-comparison}
We briefly illustrate how the convex hull volume compares to the volumes of the three double McCormick relaxations for the polytope $\tilde{P}$ (where the variables have been ordered so that $(\Omega')$ holds). Following \cite{SpeakmanLee17}, for $i = 1, 2, 3$, let $\mathcal{P}_i$ denote the double McCormick relaxation of $\tilde{P}$ obtained by first grouping the pair of variables $x_j$ and $x_k$ with $j$ and $k$ different from $i$. That is, $\mathcal{P}_i$ is constructed by introducing an auxiliary variable $w_{jk} = x_j x_k$ and applying the McCormick inequalities first to $x_j x_k$ and then to $w_{jk} x_i$. Computing the four-dimensional volume of each polytope using \texttt{scipy.spatial.ConvexHull} \citep{SciPy2020}, we obtain
\begin{equation*}
    \vol{4}{\mathcal{P}_3} \approx 1169.78, \quad 
    \vol{4}{\mathcal{P}_2} \approx 1180.95, \quad 
    \vol{4}{\mathcal{P}_1} \approx 1184.00,
\end{equation*}
each exceeding $\vol{4}{\tilde{P}} = 960$ as expected. Closed-form volume formulae for $\mathcal{P}_1, \mathcal{P}_2, \mathcal{P}_3$ over general  box domains, extending the nonnegative-bounds results of \cite{SpeakmanLee17}, remain an open question; see Section~\ref{sec:conclusion}.

\section{Conclusions and further work}
\label{sec:conclusion}
In this paper, we derived closed-form volume formulae for the convex hull of the graph of a trilinear monomial over general box domains. Under the nonnegativity assumption on our interval centers (Lemma~\ref{force_positivity}), and the ordering~\eqref{eq:omegaprime}, the volume divides into six cases determined by the center-to-half-length ratios $c_i/\ell_i$. These cases reflect the different geometric configurations that arise with mixed-sign domains.

Our work extends existing volume results to general box domains, completing the convex hull analysis initiated by \cite{SpeakmanLee17} and subsequently reproved by \cite{SpeakmanAverkov18}. While we have completed the volume analysis for the convex hull of the trilinear monomial graph, volume formulae for natural relaxations of the convex hull over general domains, such as the three alternative convexifications obtained via an iterative application of the McCormick inequalities \citep{McCormick76}, remain an open research question. Currently, such formulae exist only for nonnegative bounds \citep{SpeakmanLee17}, where under the ordering condition \eqref{eq:omega}, the three convexifications can be ranked according to their volume. 

Interesting directions for future work include deriving closed-form volume formulae for the alternative McCormick-based relaxations over general domains, and determining how they rank by tightness. It remains open whether the ranking observed for nonnegative bounds \citep{SpeakmanLee17} persists over general domains, or if mixed-sign configurations alter the relative tightness of these relaxations. Beyond this, building on the computations in \cite{SpeakmanYuLee17} and \cite{LSS22} to explore connections to overall global optimization performance represents an important avenue for future research. A further natural direction is the extension of branching-variable and branching-point selection results to general box domains, using the case-by-case volume formulae of the present paper to build on the analysis of \cite{SpeakmanLee18} for the nonnegative orthant.  Finally, extending these comparisons to higher-degree monomials remains a compelling, if challenging, open direction.

\section{Acknowledgments}{The authors would like to express their sincere gratitude to Prof. Dr. Gennadiy Averkov for his insightful contributions to this work, including his valuable suggestion to characterize box domains via their centers and half-lengths. E. Speakman further wishes to thank Dr. Averkov for the numerous conversations in which he generously shared his expertise, patiently and enthusiastically explaining important nuances of mixed volume theory.

In addition, the authors gratefully acknowledge summer funding for  L. Makhoul from the University of Colorado Denver Department of Mathematical and Statistical Sciences.

\bibliographystyle{abbrvnat} 
\bibliography{refArXivfinal}

\appendix

\section*{\hypertarget{appendix}{Appendix}}\label{sec:Appendix}

Here we present the technical lemmas required for our main result.  To begin, we recall the following definitions from Sections \ref{sec:geometric} and \ref{sec:prelim}, which we restate for readability. 

The extreme points of $\mathcal{P}_H^3(\mathbf{c},\Bell)$:

 {\small \begin{align*} 
&\mathbf{v}_1 \coloneq \begin{bmatrix}
      (c_1-\ell_1)(c_2-\ell_2)(c_3-\ell_3) \\
      (c_1-\ell_1)\\
      (c_2-\ell_2)\\
      (c_3-\ell_3)
\end{bmatrix},\; \mathbf{v}_2 \coloneq\begin{bmatrix}
      (c_1+\ell_1)(c_2-\ell_2)(c_3-\ell_3) \\
      (c_1+\ell_1)\\
      (c_2-\ell_2)\\
      (c_3-\ell_3)
\end{bmatrix}, \; \mathbf{v}_3\coloneq\begin{bmatrix}
      (c_1-\ell_1)(c_2+\ell_2)(c_3-\ell_3) \\
      (c_1-\ell_1)\\
      (c_2+\ell_2)\\
      (c_3-\ell_3)
\end{bmatrix},\\[0.5em] &\mathbf{v}_4 \coloneq\begin{bmatrix}
      (c_1+\ell_1)(c_2+\ell_2)(c_3-\ell_3) \\
      (c_1+\ell_1)\\
      (c_2+\ell_2)\\
      (c_3-\ell_3)
\end{bmatrix}, \; \mathbf{v}_5 \coloneq\begin{bmatrix}
      (c_1-\ell_1)(c_2-\ell_2)(c_3+\ell_3) \\
      (c_1-\ell_1)\\
      (c_2-\ell_2)\\
      (c_3+\ell_3)
\end{bmatrix}, \; \mathbf{v}_6 \coloneq\begin{bmatrix}
      (c_1+\ell_1)(c_2-\ell_2)(c_3+\ell_3) \\
      (c_1+\ell_1)\\
      (c_2-\ell_2)\\
      (c_3+\ell_3)
\end{bmatrix}, \\[0.5em] &\mathbf{v}_7\coloneq\begin{bmatrix}
      (c_1-\ell_1)(c_2+\ell_2)(c_3+\ell_3) \\
      (c_1-\ell_1)\\
      (c_2+\ell_2)\\
      (c_3+\ell_3)
\end{bmatrix},\; \mathbf{v}_8 \coloneq\begin{bmatrix}
      (c_1+\ell_1)(c_2+\ell_2)(c_3+\ell_3) \\
      (c_1+\ell_1)\\
       (c_2+\ell_2)\\
      (c_3+\ell_3)
\end{bmatrix}.\end{align*}}

The facet normal vectors:
{\small \begin{align*} 
&\mathbf{u}_1 \coloneq  \begin{bmatrix}
      1 \\
      -(c_2-\ell_2)(c_3-\ell_3)\\
      -(c_1+\ell_1)(c_3-\ell_3)\\
      0
\end{bmatrix},\; \mathbf{u}_2 \coloneq\begin{bmatrix}
      1 \\
      -(c_2+\ell_2)(c_3-\ell_3)\\
      -(c_1-\ell_1)(c_3-\ell_3)\\
      0
\end{bmatrix}, \; \mathbf{u}_3\coloneq\begin{bmatrix}
      -1 \\
      (c_2+\ell_2)(c_3-\ell_3)\\
      (c_1+\ell_1)(c_3-\ell_3)\\
      0
\end{bmatrix}, \\[0.5em] &\mathbf{u}_4 \coloneq\begin{bmatrix}
      -1\\
      (c_2-\ell_2)(c_3-\ell_3)\\
      (c_1-\ell_1)(c_3-\ell_3)\\
      0
\end{bmatrix}, \; \mathbf{u}_5 \coloneq\begin{bmatrix}
      1 \\
      -(c_2-\ell_2)(c_3+\ell_3)\\
      -(c_1+\ell_1)(c_3+\ell_3)\\
      0
\end{bmatrix}, \; \mathbf{u}_6 \coloneq\begin{bmatrix}
      1 \\
      -(c_2+\ell_2)(c_3+\ell_3)\\
      -(c_1-\ell_1)(c_3+\ell_3)\\
      0
\end{bmatrix}, \\[0.5em] &\mathbf{u}_7\coloneq\begin{bmatrix}
      -1 \\
      (c_2+\ell_2)(c_3+\ell_3)\\
      (c_1+\ell_1)(c_3+\ell_3)\\
      0
\end{bmatrix},\; \mathbf{u}_8 \coloneq\begin{bmatrix}
      -1\\
      (c_2-\ell_2)(c_3+\ell_3)\\
      (c_1-\ell_1)(c_3+\ell_3)\\
      0
\end{bmatrix}.\end{align*}}

The set of outer facet normals of $Q$:
$$\displaystyle
    \mathcal{U}(Q)\coloneq \begin{cases} 2\ell_1\ell_2\left\{\mathbf{u}_1,\mathbf{u}_2,\mathbf{u}_3,\mathbf{u}_4\right\}, & \frac{c_3}{\ell_3} \mathcolor{teal}{>} 1, \\ 2\ell_1\ell_2\left\{-\mathbf{u}_1,-\mathbf{u}_2,-\mathbf{u}_3,-\mathbf{u}_4\right\}, & \frac{c_3}{\ell_3} < 1,  \end{cases}$$
and the set of outer facet normals of $R$:
$$ \mathcal{U}(R) \coloneq2\ell_1\ell_2\left\{\mathbf{u}_5,\mathbf{u}_6,\mathbf{u}_7,\mathbf{u}_8\right\}.$$
Furthermore, we note that because $Q$ lies in the hyperplane $x_3 = c_3 - \ell_3$ and $R$ lies in the hyperplane $x_3 = c_3 + \ell_3$, their outer facet normals  $\mathbf{u}_i$ are 4-dimensional vectors with a zero component in the $x_3$ direction (the fourth component). Consequently, when computing $\mathbf{v}_j^T \mathbf{u}_i$ for any extreme point $\mathbf{v}_j$, the contribution from the $x_3$ component vanishes, and the inner product depends only on the first three coordinates.

We also recall that we make the nonnegativity assumption $c_i \geq 0$ for $i = 1,2,3$, and by definition $\ell_i >0$  for $i = 1,2,3$.  Moreover, without loss of generality, we assume the ordering \eqref{eq:omegaprime}: 
\begin{align*}
        \frac{c_1}{\ell_1}\leq \frac{c_2}{\ell_2}\leq \frac{c_3}{\ell_3}. 
\end{align*}
Finally, we recall from Proposition \ref{prop:V(Q,Q,R)} that $z_i^+$ corresponds to $\frac{c_3}{\ell_3}>1$, and $z_i^-$ corresponds to $\frac{c_3}{\ell_3}<1$, for $i=1,2,3,4$.

We are now ready to state and prove Lemmas \ref{z_1}--\ref{z_8}.


\begin{lemma}\label{z_1}
The relevant inner products are given by:
\begin{align*}
    \mathbf{v}_5^T\mathbf{u}_1 &=  \left(c_2-\ell_2\right) \left(c_1 \left(3\ell_3-c_3\right)-\ell_1 \left(c_3+\ell_3\right)\right),\\
    \mathbf{v}_6^T\mathbf{u}_1 &=  -\left(c_1+\ell_1\right) \left(c_2-\ell_2\right) \left(c_3-3 \ell_3\right),\\
    \mathbf{v}_7^T\mathbf{u}_1 &=  c_1 \left(\ell_2 \left(c_3+\ell_3\right)-c_2 \left(c_3-3 \ell_3\right)\right)-\ell_1 \left(\ell_2 \left(3 c_3-\ell_3\right)+c_2 \left(c_3+\ell_3\right)\right),\\
    \mathbf{v}_8^T\mathbf{u}_1 &= \left(c_1+\ell_1\right) \left(\ell_2 \left(c_3+\ell_3\right)-c_2 \left(c_3-3 \ell_3\right)\right).
\end{align*}

Then
$$z^+_1 \coloneq \max\{\mathbf{v}_5^T\mathbf{u}_1,\mathbf{v}_6^T\mathbf{u}_1,\mathbf{v}_7^T\mathbf{u}_1,\mathbf{v}_8^T\mathbf{u}_1\} = 
    \mathbf{v}_8^T\mathbf{u}_1,$$ and
$$z^-_1 \coloneq \max\{-\mathbf{v}_5^T\mathbf{u}_1,-\mathbf{v}_6^T\mathbf{u}_1,-\mathbf{v}_7^T\mathbf{u}_1,-\mathbf{v}_8^T\mathbf{u}_1\} = \begin{cases}
    -\mathbf{v}_7^T\mathbf{u}_1, &\text{if }\frac{c_2}{\ell_2} +  \frac{c_3}{\ell_3} \geq 1 + \frac{c_1}{\ell_1},\\
    -\mathbf{v}_6^T\mathbf{u}_1,&\text{if } \frac{c_2}{\ell_2} +  \frac{c_3}{\ell_3} < 1 + \frac{c_1}{\ell_1}.
\end{cases} $$
\end{lemma}
\begin{proof}
The inner product formulas follow from direct computation using the definitions of $\mathbf{v}_j$ and $\mathbf{u}_1$ given in the Appendix preamble.

We show that $\mathbf{v}_8^T\mathbf{u}_1 \geq \mathbf{v}_j^T\mathbf{u}_1$ for all $j \in \{5,6,7\}$ by computing the pairwise differences:
\begin{align*}
\mathbf{v}_8^T\mathbf{u}_1 - \mathbf{v}_5^T\mathbf{u}_1 &= 4\ell_3(\ell_2 c_1 + \ell_1 c_2) \geq 0,\\
\mathbf{v}_8^T\mathbf{u}_1 - \mathbf{v}_6^T\mathbf{u}_1 &= 4\ell_2\ell_3(c_1 + \ell_1) \geq 0, \\
\mathbf{v}_8^T\mathbf{u}_1 - \mathbf{v}_7^T\mathbf{u}_1 &= 4\ell_1(\ell_3 c_2 + \ell_2 c_3) \geq 0. 
\end{align*}

Each inequality holds because of our nonnegativity assumption. Thus $\mathbf{v}_8^T\mathbf{u}_1$ is the maximum of the four values.

We now determine the maximum of $\{-\mathbf{v}_5^T\mathbf{u}_1, -\mathbf{v}_6^T\mathbf{u}_1, -\mathbf{v}_7^T\mathbf{u}_1, -\mathbf{v}_8^T\mathbf{u}_1\}$. Computing pairwise differences we obtain:
\begin{align*}
-\mathbf{v}_7^T\mathbf{u}_1 - (-\mathbf{v}_5^T\mathbf{u}_1) &= 4\ell_2(\ell_1 c_3 - \ell_3 c_1)\geq 0, \\
 -\mathbf{v}_7^T\mathbf{u}_1 - (-\mathbf{v}_8^T\mathbf{u}_1) &= 4\ell_1(\ell_3 c_2 + \ell_2 c_3)  \geq 0,\\
-\mathbf{v}_7^T\mathbf{u}_1 - (-\mathbf{v}_6^T\mathbf{u}_1) &= 4\bigl(\ell_2(\ell_1 c_3 - \ell_3 c_1) + \ell_1\ell_3(c_2 - \ell_2)\bigr).
\end{align*}

The first inequality holds by \eqref{eq:omegaprime}. The second inequality holds because of nonnegativity.  For the final expression, we have
\begin{align*}
&\quad 4\bigl(\ell_2(\ell_1 c_3 - \ell_3 c_1) + \ell_1\ell_3(c_2 - \ell_2)\bigr) &\geq 0 \\
\Longleftrightarrow &\quad \ell_1\ell_2 c_3 - \ell_2\ell_3 c_1 + \ell_1\ell_3 c_2 - \ell_1\ell_2\ell_3 &\geq 0\\
\Longleftrightarrow &\quad \frac{c_3}{\ell_3} - \frac{c_1}{\ell_1} + \frac{c_2}{\ell_2} - 1 &\geq 0
\end{align*}

Thus, if $\frac{c_2}{\ell_2} +  \frac{c_3}{\ell_3} \geq 1 + \frac{c_1}{\ell_1}$, then $z_1^- = -\mathbf{v}_7^T\mathbf{u}_1$.  If $\frac{c_2}{\ell_2} +  \frac{c_3}{\ell_3} < 1 + \frac{c_1}{\ell_1}$, then $-\mathbf{v}_6^T\mathbf{u}_1 > -\mathbf{v}_7^T\mathbf{u}_1$, and we have $z_1^- = -\mathbf{v}_6^T\mathbf{u}_1$.  
\end{proof}

\begin{lemma}\label{z_2}
The relevant inner products are given by:
\begin{align*}
    \mathbf{v}_5^T\mathbf{u}_2 &=  \left(c_1-\ell_1\right) \left(c_2 \left(3\ell_3-c_3\right)-\ell_2 \left(c_3+\ell_3\right)\right),\\
    \mathbf{v}_6^T\mathbf{u}_2 &= \ell_1 \left(\ell_2 \left(\ell_3-3 c_3\right)+c_2 \left(c_3+\ell_3\right)\right)-c_1 \left(\ell_2 \left(c_3+\ell_3\right)+c_2 \left(c_3-3 \ell_3\right)\right),\\
    \mathbf{v}_7^T\mathbf{u}_2 &= -\left(c_1-\ell_1\right) \left(c_2+\ell_2\right) \left(c_3-3 \ell_3\right),\\
    \mathbf{v}_8^T\mathbf{u}_2 &= \left(c_2+\ell_2\right) \left(\ell_1 \left(c_3+\ell_3\right)-c_1 \left(c_3-3 \ell_3\right)\right).\\
\end{align*}

Then
$$z^+_2 \coloneq \max\{\mathbf{v}_5^T\mathbf{u}_2,\mathbf{v}_6^T\mathbf{u}_2,\mathbf{v}_7^T\mathbf{u}_2,\mathbf{v}_8^T\mathbf{u}_2\}=
    \mathbf{v}_8^T\mathbf{u}_2,$$ and
$$z^-_2 \coloneq \max\{-\mathbf{v}_5^T\mathbf{u}_2,-\mathbf{v}_6^T\mathbf{u}_2,-\mathbf{v}_7^T\mathbf{u}_2,-\mathbf{v}_8^T\mathbf{u}_2\}=
    -\mathbf{v}_7^T\mathbf{u}_2.$$
\end{lemma}

\begin{proof}
The inner product formulas follow from direct computation using the definitions of $\mathbf{v}_j$ and $\mathbf{u}_2$.
We show that $\mathbf{v}_8^T\mathbf{u}_2 \geq \mathbf{v}_j^T\mathbf{u}_2$ for all $j \in \{5,6,7\}$ by computing the pairwise differences:
\begin{align*}
\mathbf{v}_8^T\mathbf{u}_2 - \mathbf{v}_5^T\mathbf{u}_2 &= 4\ell_3(\ell_2 c_1 + \ell_1 c_2) \geq 0,\\
\mathbf{v}_8^T\mathbf{u}_2 - \mathbf{v}_6^T\mathbf{u}_2 &= 4\ell_2(\ell_1 c_3 + \ell_3 c_1) \geq 0,\\
\mathbf{v}_8^T\mathbf{u}_2 - \mathbf{v}_7^T\mathbf{u}_2 &= 4\ell_1\ell_3(c_2 + \ell_2) \geq 0.
\end{align*}

Each inequality holds because of our nonnegativity assumption. Thus $\mathbf{v}_8^T\mathbf{u}_2$ is the maximum of the four values.

We now determine the maximum of $\{-\mathbf{v}_5^T\mathbf{u}_2, -\mathbf{v}_6^T\mathbf{u}_2, -\mathbf{v}_7^T\mathbf{u}_2, -\mathbf{v}_8^T\mathbf{u}_2\}$. Computing pairwise differences we obtain:
\begin{align*}
-\mathbf{v}_7^T\mathbf{u}_2 - (-\mathbf{v}_5^T\mathbf{u}_2) &= 4\ell_2\ell_3(\ell_1 - c_1) \geq 0,\\
-\mathbf{v}_7^T\mathbf{u}_2 - (-\mathbf{v}_8^T\mathbf{u}_2) &= 4\ell_1\ell_3(c_2 + \ell_2) \geq 0,\\
-\mathbf{v}_7^T\mathbf{u}_2 - (-\mathbf{v}_6^T\mathbf{u}_2) &= 4\bigl(\ell_3(\ell_1(c_2 + \ell_2) - c_1\ell_2) - c_3\ell_1\ell_2\bigr)\geq 0.
\end{align*}

The first inequality holds because \eqref{eq:omegaprime} gives $\frac{c_3}{\ell_3}< 1 \Rightarrow \frac{c_1}{\ell_1}< 1$.  The second inequality holds by nonnegativity. For the third inequality, suppose for contradiction, that $-\mathbf{v}_7^T\mathbf{u}_2 - (-\mathbf{v}_6^T\mathbf{u}_2) < 0$. Then, rearranging and dividing by $\ell_1\ell_2\ell_3 > 0$ gives:
$$ 1 - \frac{c_3}{\ell_3} < \frac{c_1}{\ell_1} - \frac{c_2}{\ell_2}. $$

By assumption we have $\frac{c_3}{\ell_3} < 1$, and by \eqref{eq:omegaprime}, $\frac{c_1}{\ell_1} \leq \frac{c_2}{\ell_2}$, which leads to
\[
(0 <)\; 1 - \frac{c_3}{\ell_3} < \frac{c_1}{\ell_1} - \frac{c_2}{\ell_2} \;(\leq 0),
\]
a contradiction. Thus, $-\mathbf{v}_7^T\mathbf{u}_2 \geq -\mathbf{v}_6^T\mathbf{u}_2$, and we conclude $z_2^- = -\mathbf{v}_7^T\mathbf{u}_2$.  
\end{proof}

\begin{lemma}\label{z_3}
The relevant inner products are given by:
\begin{align*}
    \mathbf{v}_5^T\mathbf{u}_3 &= \ell _1 \left(\ell _2 \left(\ell _3-3 c_3\right)+c_2 \left(c_3+\ell _3\right)\right)+c_1 \left(\ell _2 \left(c_3+\ell _3\right)+c_2 \left(c_3-3 \ell _3\right)\right), \\
    \mathbf{v}_6^T\mathbf{u}_3 &= \left(c_1+\ell_1\right) \left(\ell_2 \left(c_3+\ell_3\right)+c_2 \left(c_3-3 \ell_3\right)\right),\\
    \mathbf{v}_7^T\mathbf{u}_3 &= \left(c_2+\ell_2\right) \left(\ell_1 \left(c_3+\ell_3\right)+c_1 \left(c_3-3 \ell_3\right)\right),\\
     \mathbf{v}_8^T\mathbf{u}_3 &=  \left(c_1+\ell_1\right) \left(c_2+\ell_2\right) \left(c_3-3 \ell_3\right).\\
\end{align*}

Then
$$z^+_3 \coloneq \max\{\mathbf{v}_5^T\mathbf{u}_3,\mathbf{v}_6^T\mathbf{u}_3,\mathbf{v}_7^T\mathbf{u}_3,\mathbf{v}_8^T\mathbf{u}_3\}=
    \mathbf{v}_7^T\mathbf{u}_3,$$ and
$$z^-_3 \coloneq \max\{-\mathbf{v}_5^T\mathbf{u}_3,-\mathbf{v}_6^T\mathbf{u}_3,-\mathbf{v}_7^T\mathbf{u}_3,-\mathbf{v}_8^T\mathbf{u}_3\}= 
    -\mathbf{v}_8^T\mathbf{u}_3.$$
\end{lemma}

\begin{proof}
The inner product formulas follow from direct computation using the definitions of $\mathbf{v}_j$ and $\mathbf{u}_3$.
We show that $\mathbf{v}_7^T\mathbf{u}_3 \geq \mathbf{v}_j^T\mathbf{u}_3$ for all $j \in \{5,6,8\}$ by computing the pairwise differences:
\begin{align*}
\mathbf{v}_7^T\mathbf{u}_3 - \mathbf{v}_5^T\mathbf{u}_3 &= 4\ell_2(\ell_1 c_3 - \ell_3 c_1) \geq 0,\\
\mathbf{v}_7^T\mathbf{u}_3 - \mathbf{v}_6^T\mathbf{u}_3 &= 4\ell_3(\ell_1 c_2 - \ell_2 c_1) \geq 0, \\
\mathbf{v}_7^T\mathbf{u}_3 - \mathbf{v}_8^T\mathbf{u}_3 &= 4\ell_1\ell_3(c_2 + \ell_2) \geq 0.
\end{align*}

The first two inequalities hold by \eqref{eq:omegaprime}. The third inequality holds because of nonnegativity.  Thus $\mathbf{v}_7^T\mathbf{u}_3$ is the maximum of the four values.

We now determine the maximum of $\{-\mathbf{v}_5^T\mathbf{u}_3, -\mathbf{v}_6^T\mathbf{u}_3, -\mathbf{v}_7^T\mathbf{u}_3, -\mathbf{v}_8^T\mathbf{u}_3\}$. Computing pairwise differences:

\begin{align*}
-\mathbf{v}_5^T\mathbf{u}_3 - (-\mathbf{v}_7^T\mathbf{u}_3) &= 4\ell_2(\ell_1 c_3 - \ell_3 c_1) \geq 0, \\
-\mathbf{v}_5^T\mathbf{u}_3 - (-\mathbf{v}_6^T\mathbf{u}_3) &= 4\ell_1(\ell_2 c_3 - \ell_3 c_2) \geq 0,\\
-\mathbf{v}_8^T\mathbf{u}_3 - (-\mathbf{v}_5^T\mathbf{u}_3) &= 4\ell_2\ell_3(c_1 + \ell_1) - 4\ell_1(\ell_2 c_3 - \ell_3 c_2) \geq 0.
\end{align*}

The first two inequalities hold by \eqref{eq:omegaprime}. For the third, by way of contradiction, suppose $-\mathbf{v}_8^T\mathbf{u}_3 - (-\mathbf{v}_5^T\mathbf{u}_3) < 0$. Rearranging, we obtain 
$$\frac{c_3}{\ell_3} - 1 > \frac{c_1}{\ell_1} + \frac{c_2}{\ell_2}.$$ 
By assumption we have  $\frac{c_3}{\ell_3} - 1 < 0$ and $\frac{c_1}{\ell_1} + \frac{c_2}{\ell_2} \geq 0$, which leads to
$$(0 >) \; \frac{c_3}{\ell_3} - 1 > \frac{c_1}{\ell_1} + \frac{c_2}{\ell_2} \; (>0),$$ 
a contradiction.  Thus $-\mathbf{v}_8^T\mathbf{u}_3 \geq -\mathbf{v}_5^T\mathbf{u}_3$, and we conclude $z_3^- = -\mathbf{v}_8^T\mathbf{u}_3$.  

\end{proof}

\begin{lemma}\label{z_4}
The relevant inner products are given by:
\begin{align*}
   \mathbf{v}_5^T\mathbf{u}_4 &=  \left(c_1-\ell_1\right) \left(c_2-\ell_2\right) \left(c_3-3 \ell_3\right),\\
    \mathbf{v}_6^T\mathbf{u}_4 &=  \left(c_2-\ell_2\right) \left(c_1 \left(c_3-3 \ell_3\right)-\ell_1 \left(c_3+\ell_3\right)\right),\\
     \mathbf{v}_7^T\mathbf{u}_4 &= \left(c_1-\ell_1\right) \left(c_2 \left(c_3-3 \ell_3\right)-\ell_2 \left(c_3+\ell_3\right)\right),\\
      \mathbf{v}_8^T\mathbf{u}_4 &=  c_1 \left(c_2 \left(c_3-3 \ell_3\right)-\ell_2 \left(c_3+\ell_3\right)\right)-\ell_1 \left(\ell_2 \left(3 c_3-\ell_3\right)+c_2 \left(c_3+\ell_3\right)\right).\\
\end{align*}

Then
$$z^+_4 \coloneq \max\{\mathbf{v}_5^T\mathbf{u}_4,\mathbf{v}_6^T\mathbf{u}_4,\mathbf{v}_7^T\mathbf{u}_4,\mathbf{v}_8^T\mathbf{u}_4\}= \begin{cases}
        \mathbf{v}_5^T\mathbf{u}_4, &\text{if }\frac{c_1}{\ell_1}\geq1,\\
        \mathbf{v}_7^T\mathbf{u}_4,&\text{if }\frac{c_1}{\ell_1}<1,
\end{cases}
$$
and
$$z^-_4 \coloneq \max\{\mathcolor{teal}{-}\mathbf{v}_5^T\mathbf{u}_4,\mathcolor{teal}{-}\mathbf{v}_6^T\mathbf{u}_4,\mathcolor{teal}{-}\mathbf{v}_7^T\mathbf{u}_4,\mathcolor{teal}{-}\mathbf{v}_8^T\mathbf{u}_4\} = \begin{cases}
        -\mathbf{v}_8^T\mathbf{u}_4, &\text{if }\frac{c_1}{\ell_1} +\frac{c_2}{\ell_2} + \frac{c_3}{\ell_3}\geq1,\\
    -\mathbf{v}_5^T\mathbf{u}_4,&\text{if } \frac{c_1}{\ell_1} +\frac{c_2}{\ell_2} + \frac{c_3}{\ell_3}<1.
\end{cases}
$$
\end{lemma}

\begin{proof}
The inner product formulas follow from direct computation using the definitions of $\mathbf{v}_j$ and $\mathbf{u}_4$.
We compute the pairwise differences:
\begin{align*}
\mathbf{v}_7^T\mathbf{u}_4 - \mathbf{v}_6^T\mathbf{u}_4 &= 4\ell_3(\ell_1 c_2 - \ell_2 c_1) \geq 0, \\
\mathbf{v}_7^T\mathbf{u}_4 - \mathbf{v}_8^T\mathbf{u}_4 &= 4\ell_1(\ell_3 c_2 + \ell_2 c_3) \geq 0,\\
\mathbf{v}_5^T\mathbf{u}_4 - \mathbf{v}_7^T\mathbf{u}_4 &= 4\ell_2\ell_3(c_1 - \ell_1).
\end{align*}

The first inequality holds by \eqref{eq:omegaprime}. The second inequality holds because of nonnegativity. For the third expression, consider two cases. If $\frac{c_1}{\ell_1} \geq 1$, then $\mathbf{v}_5^T\mathbf{u}_4 - \mathbf{v}_7^T\mathbf{u}_4 \geq 0$ and $z^+_4 = \mathbf{v}_5^T\mathbf{u}_4$. Otherwise, $\mathbf{v}_5^T\mathbf{u}_4 - \mathbf{v}_7^T\mathbf{u}_4 < 0$ and $z^+_4 = \mathbf{v}_7^T\mathbf{u}_4$.

We now determine the maximum of $\{-\mathbf{v}_5^T\mathbf{u}_4, -\mathbf{v}_6^T\mathbf{u}_4, -\mathbf{v}_7^T\mathbf{u}_4, -\mathbf{v}_8^T\mathbf{u}_4\}$. Computing pairwise differences:
\begin{align*}
-\mathbf{v}_8^T\mathbf{u}_4 - (-\mathbf{v}_6^T\mathbf{u}_4) &= 4\ell_2(c_3\ell_1 + c_1\ell_3) \geq 0,\\
-\mathbf{v}_8^T\mathbf{u}_4 - (-\mathbf{v}_7^T\mathbf{u}_4) &= 4\ell_1(c_3\ell_2 + c_2\ell_3) \geq 0, \\
-\mathbf{v}_8^T\mathbf{u}_4 - (-\mathbf{v}_5^T\mathbf{u}_4) &= 4(c_3\ell_1\ell_2 + \ell_3(c_2\ell_1 + \ell_2(c_1 - \ell_1))).
\end{align*}

The first and second inequalities hold because of nonnegativity. For the third expression, we have
\begin{align*}
&\quad 4(c_3\ell_1\ell_2 + \ell_3(c_2\ell_1 + \ell_2(c_1 - \ell_1))) &\geq 0\\
\Longleftrightarrow &\quad c_3\ell_1\ell_2 + c_2\ell_1\ell_3 + c_1\ell_2\ell_3 - \ell_1\ell_2\ell_3 &\geq 0\\
\Longleftrightarrow &\quad \frac{c_1}{\ell_1} + \frac{c_2}{\ell_2} + \frac{c_3}{\ell_3} - 1 &\geq 0
\end{align*}

Thus, if $\frac{c_1}{\ell_1} + \frac{c_2}{\ell_2} + \frac{c_3}{\ell_3} \geq 1$, we have $-\mathbf{v}_8^T\mathbf{u}_4 \geq -\mathbf{v}_5^T\mathbf{u}_4$, and $z_4^- = -\mathbf{v}_8^T\mathbf{u}_4$. Otherwise,  $-\mathbf{v}_5^T\mathbf{u}_4>-\mathbf{v}_8^T\mathbf{u}_4$, and $z_4^- = -\mathbf{v}_5^T\mathbf{u}_4$.  
\end{proof}

\begin{lemma}\label{z_5}
The relevant inner products are given by:
    \begin{align*}
        \mathbf{v}_1^T\mathbf{u}_5 &=  \left(c_2-\ell_2\right) \left(\ell_1 \left(\ell_3-c_3\right) - c_1 \left(c_3+3 \ell_3\right)\right),\\
        \mathbf{v}_2^T\mathbf{u}_5 &= -\left(c_1+\ell_1\right) \left(c_2-\ell_2\right) \left(c_3+3 \ell_3\right),\\
        \mathbf{v}_3^T\mathbf{u}_5 &= \ell_1 \left(c_2 \left(\ell_3-c_3\right)-\ell_2 \left(3 c_3+\ell_3\right)\right)-c_1 \left(\ell_2 \left(\ell_3-c_3\right)+c_2 \left(c_3+3 \ell_3\right)\right),\\
         \mathbf{v}_4^T\mathbf{u}_5 &= \left(c_1+\ell_1\right) \left(\ell_2 \left(c_3-\ell_3\right) -c_2 \left(c_3+3 \ell_3\right)\right).
    \end{align*}
    Then
    $$z_5 \coloneq \max\{\mathbf{v}_1^T\mathbf{u}_5,\mathbf{v}_2^T\mathbf{u}_5,\mathbf{v}_3^T\mathbf{u}_5,\mathbf{v}_4^T\mathbf{u}_5\} = \begin{cases}
    \mathbf{v}_1^T\mathbf{u}_5,&\text{if $\frac{c_2}{\ell_2}\geq1$},\\
    \mathbf{v}_2^T\mathbf{u}_5,&\text{if $\frac{c_2}{\ell_2}<1$}.
\end{cases}$$
\end{lemma}

\begin{proof}
The inner product formulas follow from direct computation using the definitions of $\mathbf{v}_j$ and $\mathbf{u}_5$.
We compute the pairwise differences:
\begin{align*}
\mathbf{v}_1^T\mathbf{u}_5 - \mathbf{v}_3^T\mathbf{u}_5 &= 4\ell_2(c_1\ell_3 + \ell_1c_3) \geq 0,\\
\mathbf{v}_1^T\mathbf{u}_5 - \mathbf{v}_4^T\mathbf{u}_5 &= 4\ell_3( c_1\ell_2 + \ell_1 c_2) \geq 0,\\
\mathbf{v}_1^T\mathbf{u}_5 - \mathbf{v}_2^T\mathbf{u}_5 &= 4\ell_1\ell_3(c_2 - \ell_2).
\end{align*}

The first two inequalities hold because of nonnegativity.  For the third expression, consider two cases. If $\frac{c_2}{\ell_2} \geq 1$, then $\mathbf{v}_1^T\mathbf{u}_5 - \mathbf{v}_2^T\mathbf{u}_5 \geq 0$ and $z_5 = \mathbf{v}_1^T\mathbf{u}_5$. Otherwise, $\mathbf{v}_1^T\mathbf{u}_5 - \mathbf{v}_2^T\mathbf{u}_5 < 0$ and $z_5 = \mathbf{v}_2^T\mathbf{u}_5$.  
\end{proof}


\begin{lemma}\label{z_6}
The relevant inner products are given by:
    \begin{align*}
     \mathbf{v}_1^T\mathbf{u}_6 &= \left(c_1-\ell_1\right) \left(\ell_2 \left(\ell_3-c_3\right) - c_2 \left(c_3+3 \ell_3\right)\right),\\
   \mathbf{v}_2^T\mathbf{u}_6 &= \ell_1 \left(c_3 \left(c_2-3 \ell_2\right)-\ell_3 \left(c_2+\ell_2\right)\right)-c_1 \left(\ell_2 \left(c_3-\ell_3\right)+c_2 \left(c_3+3 \ell_3\right)\right),\\
   \mathbf{v}_3^T\mathbf{u}_6 &=  -\left(c_1-\ell_1\right) \left(c_2+\ell_2\right) \left(c_3+3 \ell_3\right),\\
    \mathbf{v}_4^T\mathbf{u}_6 &= \left(c_2+\ell_2\right) \left(\ell_1 \left(c_3-\ell_3\right)-c_1 \left(c_3+3 \ell_3\right)\right).\\
\end{align*}

Then
$$z_6 \coloneq \max\{\mathbf{v}_1^T\mathbf{u}_6,\mathbf{v}_2^T\mathbf{u}_6,\mathbf{v}_3^T\mathbf{u}_6,\mathbf{v}_4^T\mathbf{u}_6\}= \begin{cases}
    \mathbf{v}_1^T\mathbf{u}_6,&\text{if $\frac{c_1}{\ell_1}\geq1$},\\
    \mathbf{v}_3^T\mathbf{u}_6,&\text{if $\frac{c_1}{\ell_1}<1$}.
\end{cases}$$
\end{lemma}

\begin{proof}
The inner product formulas follow from direct computation using the definitions of $\mathbf{v}_j$ and $\mathbf{u}_6$.
We compute the pairwise differences:
\begin{align*}
\mathbf{v}_1^T\mathbf{u}_6 - \mathbf{v}_2^T\mathbf{u}_6 &= 4\ell_1(c_2\ell_3 + \ell_2 c_3) \geq 0,\\
\mathbf{v}_1^T\mathbf{u}_6 - \mathbf{v}_4^T\mathbf{u}_6 &= 4\ell_3(\ell_2 c_1 + \ell_1 c_2) \geq 0,\\
\mathbf{v}_1^T\mathbf{u}_6 - \mathbf{v}_3^T\mathbf{u}_6 &= 4\ell_2\ell_3(c_1 - \ell_1).
\end{align*}

The first and second inequalities hold because of nonnegativity. For the third expression, consider two cases. If $\frac{c_1}{\ell_1} \geq 1$, then $\mathbf{v}_1^T\mathbf{u}_6 - \mathbf{v}_3^T\mathbf{u}_6 \geq 0$ and $z_6 = \mathbf{v}_1^T\mathbf{u}_6$. Otherwise, $\mathbf{v}_1^T\mathbf{u}_6 - \mathbf{v}_3^T\mathbf{u}_6 < 0$ and $z_6 = \mathbf{v}_3^T\mathbf{u}_6$.

\end{proof}

\begin{lemma}\label{z_7}
The relevant inner products are given by:
    \begin{align*}
       \mathbf{v}_1^T\mathbf{u}_7 &= \ell_1 \left(c_3 \left(c_2-3 \ell_2\right)-\ell_3 \left(c_2+\ell_2\right)\right)+c_1 \left(\ell_2 \left(c_3-\ell_3\right)+c_2 \left(c_3+3 \ell_3\right)\right),\\
        \mathbf{v}_2^T\mathbf{u}_7 &= \left(c_1+\ell_1\right) \left(\ell_2 \left(c_3-\ell_3\right)+c_2 \left(c_3+3 \ell_3\right)\right),\\
        \mathbf{v}_3^T\mathbf{u}_7 &= \left(c_2+\ell_2\right) \left(\ell_1 \left(c_3-\ell_3\right)+c_1 \left(c_3+3 \ell_3\right)\right),\\
        \mathbf{v}_4^T\mathbf{u}_7 &= \left(c_1+\ell_1\right) \left(c_2+\ell_2\right) \left(c_3+3 \ell_3\right).
    \end{align*}

    Then
    $$z_7\coloneq \max\{\mathbf{v}_1^T\mathbf{u}_7,\mathbf{v}_2^T\mathbf{u}_7,\mathbf{v}_3^T\mathbf{u}_7,\mathbf{v}_4^T\mathbf{u}_7\} = \mathbf{v}_4^T\mathbf{u}_7.$$
\end{lemma}

\begin{proof}
The inner product formulas follow from direct computation using the definitions of $\mathbf{v}_j$ and $\mathbf{u}_7$.
We show that $\mathbf{v}_4^T\mathbf{u}_7 \geq \mathbf{v}_j^T\mathbf{u}_7$ for all $j \in \{1,2,3\}$ by computing the pairwise differences:
\begin{align*}
\mathbf{v}_4^T\mathbf{u}_7 - \mathbf{v}_1^T\mathbf{u}_7 &= 4(\ell_2\ell_3(c_1+\ell_1) + \ell_1(c_2\ell_3+c_3\ell_2)) \geq 0,\\
\mathbf{v}_4^T\mathbf{u}_7 - \mathbf{v}_2^T\mathbf{u}_7 &= 4\ell_2\ell_3(c_1+\ell_1) \geq 0,\\
\mathbf{v}_4^T\mathbf{u}_7 - \mathbf{v}_3^T\mathbf{u}_7 &= 4\ell_1\ell_3(c_2+\ell_2) \geq 0.
\end{align*}

All three inequalities hold because of nonnegativity. Thus $z_7 = \mathbf{v}_4^T\mathbf{u}_7$.  
\end{proof}


\begin{lemma}\label{z_8}
The relevant inner products are given by:

    \begin{align*}
       \mathbf{v}_1^T\mathbf{u}_8 &= \left(c_1-\ell_1\right) \left(c_2-\ell_2\right) \left(c_3+3 \ell_3\right),\\
       \mathbf{v}_2^T\mathbf{u}_8 &= \left(c_2-\ell_2\right) \left(\ell_1 \left(\ell_3-c_3\right)+c_1 \left(c_3+3 \ell_3\right)\right),\\
      \mathbf{v}_3^T\mathbf{u}_8 &= \left(c_1-\ell_1\right) \left(\ell_2 \left(\ell_3-c_3\right)+c_2 \left(c_3+3 \ell_3\right)\right),\\
      \mathbf{v}_4^T\mathbf{u}_8 &= \ell_1 \left(c_2 \left(\ell_3-c_3\right)-\ell_2 \left(3 c_3+\ell_3\right)\right)+c_1 \left(\ell_2 \left(\ell_3-c_3\right)+c_2 \left(c_3+3 \ell_3\right)\right).
    \end{align*}

    Then
    $$z_8\coloneq \max\{\mathbf{v}_1^T\mathbf{u}_8,\mathbf{v}_2^T\mathbf{u}_8,\mathbf{v}_3^T\mathbf{u}_8,\mathbf{v}_4^T\mathbf{u}_8\}= \begin{cases}
    \mathbf{v}_2^T\mathbf{u}_8,&\text{if $\frac{c_2}{\ell_2} \geq 1$},\\
    \mathbf{v}_1^T\mathbf{u}_8,&\text{if $\frac{c_2}{\ell_2}<1$}.
\end{cases}$$
\end{lemma}

\begin{proof}
The inner product formulas follow from direct computation using the definitions of $\mathbf{v}_j$ and $\mathbf{u}_8$.
We compute the pairwise differences:
\begin{align*}
\mathbf{v}_2^T\mathbf{u}_8 - \mathbf{v}_3^T\mathbf{u}_8 &= 4\ell_3(\ell_1 c_2 - \ell_2 c_1) \geq 0, \\
\mathbf{v}_2^T\mathbf{u}_8 - \mathbf{v}_4^T\mathbf{u}_8 &= 4\ell_2(\ell_1 c_3 - \ell_3 c_1) \geq 0,\\
\mathbf{v}_2^T\mathbf{u}_8 - \mathbf{v}_1^T\mathbf{u}_8 &= 4\ell_1\ell_3(c_2 - \ell_2).
\end{align*}

The first and second inequalities hold by \eqref{eq:omegaprime}. For the third expression, consider two cases. If $\frac{c_2}{\ell_2} \geq 1$, then $\mathbf{v}_2^T\mathbf{u}_8 - \mathbf{v}_1^T\mathbf{u}_8 \geq 0$ and $z_8 = \mathbf{v}_2^T\mathbf{u}_8$. Otherwise, $\mathbf{v}_2^T\mathbf{u}_8 - \mathbf{v}_1^T\mathbf{u}_8 < 0$ and $z_8 = \mathbf{v}_1^T\mathbf{u}_8$.
  
\end{proof}

\end{document}